\documentclass[english,british]{article}
\usepackage{ae,aecompl}
\usepackage[T1]{fontenc}
\usepackage[latin9]{inputenc}
\usepackage{geometry}
\usepackage{amsmath}
\usepackage{amsthm}
\usepackage{amssymb}
\usepackage{hyperref}
\newcommand{\IE}{\mathbb{E}}
\newcommand{\IP}{\mathbb{P}}

\newcommand{\bx}{\mathbf{x}}
\newcommand{\bk}{\mathbf{k}}
\newcommand{\dd}{\;\text{d\:\!}}
\newcommand{\eh}{\frac{1}{2}}
\makeatletter
\theoremstyle{plain}
\newtheorem{thm}{\protect\theoremname}
\theoremstyle{definition}
\newtheorem{defn}[thm]{\protect\definitionname}
\theoremstyle{remark}
\newtheorem{rem}[thm]{\protect\remarkname}
\theoremstyle{plain}
\newtheorem{lem}[thm]{\protect\lemmaname}
\theoremstyle{plain}
\newtheorem{prop}[thm]{\protect\propositionname}
\theoremstyle{plain}
\newtheorem{cor}[thm]{\protect\corollaryname}

\date{}

\makeatother

\usepackage{babel}
\addto\captionsbritish{\renewcommand{\corollaryname}{Corollary}}
\addto\captionsbritish{\renewcommand{\definitionname}{Definition}}
\addto\captionsbritish{\renewcommand{\lemmaname}{Lemma}}
\addto\captionsbritish{\renewcommand{\propositionname}{Proposition}}
\addto\captionsbritish{\renewcommand{\remarkname}{Remark}}
\addto\captionsbritish{\renewcommand{\theoremname}{Theorem}}
\addto\captionsenglish{\renewcommand{\corollaryname}{Corollary}}
\addto\captionsenglish{\renewcommand{\definitionname}{Definition}}
\addto\captionsenglish{\renewcommand{\lemmaname}{Lemma}}
\addto\captionsenglish{\renewcommand{\propositionname}{Proposition}}
\addto\captionsenglish{\renewcommand{\remarkname}{Remark}}
\addto\captionsenglish{\renewcommand{\theoremname}{Theorem}}
\providecommand{\corollaryname}{Corollary}
\providecommand{\definitionname}{Definition}
\providecommand{\lemmaname}{Lemma}
\providecommand{\propositionname}{Proposition}
\providecommand{\remarkname}{Remark}
\providecommand{\theoremname}{Theorem}

\begin{document}
\global\long\def\IN{\mathbb{N}}%
\global\long\def\II{\mathbbm{1}}%
\global\long\def\IZ{\mathbb{Z}}%
\global\long\def\IQ{\mathbb{Q}}%
\global\long\def\IR{\mathbb{R}}%
\global\long\def\IC{\mathbb{C}}%
\global\long\def\IP{\mathbb{P}}%
\global\long\def\IE{\mathbb{E}}%

\title{Limit Theorems for Multi-Group Curie-Weiss Models via the Method of Moments}
\author{Werner Kirsch\thanks{FernUniversit\"at in Hagen, Germany, werner.kirsch@fernuni-hagen.de}
\;and Gabor Toth\thanks{IIMAS-UNAM, Mexico City, Mexico, gabor.toth@iimas.unam.mx}}
\maketitle
\begin{abstract}
\noindent We study a multi-group version of the mean-field or Curie-Weiss spin
model. For this model, we show how, analogously to the classical (single-group)
model, the three temperature regimes are defined. Then we use the
method of moments to determine for each regime how the vector of the
group magnetisations behaves asymptotically. Some possible applications
to social or political sciences are discussed.
\end{abstract}


\section{Definition of the Model}

The Curie-Weiss model (CWM) is usually defined for a single set of
spins or binary random variables. We have a fixed number of spins,
which assume one of two values, $\pm1$. So the state space of the
model is $\left\{ -1,1\right\} ^{N}$, where $N\in\mathbb{N}$ is
the number of spins. The spins have a tendency to align with each
other. If the majority of other spins is positive, then the conditional
probability of a given spin being positive is greater than $1/2$.
The probability measure that describes the single-group CWM for every
spin configuration $\left(x_{1},\ldots,x_{N}\right)\in\left\{ -1,1\right\} ^{N}$
is given by
\begin{equation}
\mathbb{P}\left(X_{1}=x_{1},\ldots,X_{N}=x_{N}\right)=Z^{-1}\exp\left(\frac{\beta}{2N}\left(\sum_{i=1}^{N}x_{i}\right)^{2}\right),\label{eq:P_single-group}
\end{equation}
where $\beta\geq0$ is the inverse temperature parameter.

The CWM has been extensively studied. It is named after Pierre Curie
and Pierre Weiss, who first used the mean-field approach to study
phase transitions in spin models. The CWM is also called the Husimi-Temperley
model, since it was first introduced by Husimi \cite{Husimi} and
Temperley \cite{Temperley}. Subsequently, it was discussed by Kac
\cite{Kac}, Ellis-Newman \cite{EllisN}, and many other authors. For an overview and more references see Thompson \cite{Thompson} and Ellis \cite{Ellis}. More
recently, the CWM has been used in the context of social and political
interactions. The idea of using models from statistical mechanics
to study social interactions goes back to F\"ollmer \cite{F=0000F6}.
The Curie-Weiss model specifically was first employed in \cite{BD}.
See e.g.\! \cite{CGh,GBC,WK-HO,KL,Toth,OEA,LSV,KT_opt_weights_CW} for other
applications. Multi-group versions of this model were introduced in
\cite{CG2008} and \cite{BRS}, and analysed by other authors as well
\cite{CGh,FC,FM,Co2014,LS,KLSS}. The authors of the present article
have studied special cases of the models considered in this article
with only two groups in \cite{KT1,KT2}. See the PhD thesis \cite{Toth}
of one of the authors for a detailed exposition of not only the CWM
but other models of voting behaviour.

As mentioned, this multi-group model has been analysed by several
different authors. Among the articles published, \cite{FC,KLSS} contain
many of the same limit theorems to be found in this article. The methods
employed in the earlier articles are mainly analytical in nature.
We present proofs using the method of moments which has a more combinatorial
flavour. We also prove strong laws of large numbers and a conditional central limit theorem for the low temperature regime, which to our knowledge have not been published in other articles.

The limit theorems that hold for this multi-group model are similar
to the theorems which hold for the single-group model (see Chapters
IV and V of \cite{Ellis} for an exposition). There are laws of large
numbers and central limit theorems for the vector of group magnetisations,
and these are multivariate versions of the results for the single-group
model.

Multi-group models of this type have been used to study social dynamics
and political decision making. For example, consider a population
of voters belonging to different groups or constituencies such as
a federal republic with a number of partially autonomous states. Typical examples are the Council of the European Union
and to a large extent the Electoral College of the USA. Each
group elects a representative to vote on its behalf in a council.
The voters within each group influence each other in their decisions,
more so than voters belonging to different groups do. In fact, it
may even be the case that voters belonging to different groups tend
to make contrary decisions.

To study such a situation, we define a model with $M$ different groups
of spins that potentially interact with each other in different ways.
Let
\[
(X_{11},X_{12},\ldots,X_{1N_{1}},\ldots,X_{M1},X_{M2},\ldots,X_{MN_{M}})\in\prod_{\lambda=1}^{M}\{-1,1\}^{N_{\lambda}}
~=~\{-1,1\}^{N}
\]
be a random spin configuration of the entire population. $X_{\lambda i}$
is the $i$-th spin of $N_\lambda$ in group $\lambda\in\left\{ 1,\ldots,M\right\} $.
Each group $\lambda$ consists of $N_{\lambda}$ spins, hence $\sum_{\lambda=1}^{M}N_{\lambda}=N$.
Instead of a single inverse temperature parameter, there is a coupling
matrix that describes the interactions. We will call this matrix $J:=(J_{\lambda\mu})_{\lambda,\mu=1,\ldots,M}$.
Throughout this paper we will always assume that $J$ is positive semi-definite.
Just as in the single-group model, there is a Hamiltonian function
that assigns each configuration a certain energy level. This energy
level can also be interpreted as how costly a certain situation is
in terms of the conflict between different voters.
\begin{align}\label{def:model}
\mathbb{H}(x_{11},\ldots,x_{MN_{M}}):=-\frac{1}{2N}\sum_{\lambda,\mu=1}^{M}J_{\lambda\mu}
\sum_{i=1}^{N_{\lambda}}\sum_{j=1}^{N_{\mu}}x_{\lambda i}x_{\mu j}.
\end{align}
Instead of each spin interacting with each other spin in exactly the
same way, spins in different groups $\lambda,\mu$ are coupled by
a coupling constant $J_{\lambda\mu}$. These coupling constants subsume
the inverse temperature parameter $\beta$ found in the single-group
model. We note that, depending on the signs of the coupling parameters
$J_{\lambda\mu}$, different configurations have different energy
levels assigned to them by $\mathbb{H}$. If all coupling parameters
are positive, there are two configurations that have the lowest possible
energy levels: $(-1,\ldots,-1)$ and $(1,\ldots,1)$. All other configurations
receive higher energy levels. The highest levels are those where the
spins are evenly split (or very close to it in case of odd group sizes).

We want to emphasise that the Hamiltonian \ref{def:model} -- similarly to that of the single-group model defined in \ref{eq:P_single-group} -- does not contain a term reflecting an external magnetic field. The reason for this is that applications to social sciences discussed above usually presuppose a symmetric probability measure on the space of spin configurations $\{-1,1\}^N$.

\begin{defn} A collection $(X_{11},\ldots,X_{1N_{1}},\ldots,X_{M1},\ldots,X_{MN_{M}}) $ of $\{ -1,1 \}$-valued random variables is called
an $M$-group Curie-Weiss model with coupling matrix $J$ if
the probability
of each of the $2^{N}$ spin configurations is given by
\begin{align}
\mathbb{P}\left(X_{11}=x_{11},\ldots,X_{MN_{M}}=x_{MN_{M}}\right):=Z^{-1}e^{-\mathbb{H}\left(x_{11},\ldots,x_{MN_{M}}\right)},\label{eq:Pbeta}
\end{align}
where each $x_{\lambda i}$ is in $\{-1,1\}$ and $Z$ is a normalisation
constant which depends on $N$ and $J$. The measure  $\mathbb{P}$ is called the `canonical ensemble' associated to the energy $\mathbb{H} $.
\end{defn}

As there is no external magnetic field, the measure $\mathbb{P}$
satisfies the symmetry condition
\[
\mathbb{P}\left(X_{11}=x_{11},\ldots,X_{MN_{M}}=x_{MN_{M}}\right)=\mathbb{P}\left(X_{11}=-x_{11},\ldots,X_{MN_{M}}=-x_{MN_{M}}\right)
\]
for all configurations.

In the definition of homogeneous coupling matrices below we will use the
notation
\begin{defn}
\label{def:hom_matrix}Let $(c)_{\lambda,\mu=1,\ldots,M}$ stand for
an $M\times M$ matrix with each entry equal to the constant $c$.
\end{defn}

We will deal with two classes of coupling matrices in this article:
\begin{enumerate}
\item Homogeneous coupling matrices $J=(\beta)_{\lambda,\mu=1,\ldots,M},$
where all entries are equal to the same constant $\beta\geq0$. In this case, $J$ is positive semi-definite, but not positive definite.
\item Heterogeneous coupling matrices $J=(J_{\lambda,\mu})_{\lambda,\mu=1,\ldots,M},$
which we assume to be positive definite.
\end{enumerate}

We are interested in the asymptotic behaviour of the so called magnetisations.
These are the sums of all spins belonging to each group:
\begin{align}
\boldsymbol{S} & :=(S_{1},\ldots,S_{M}):=\left(\sum_{i_{1}=1}^{N_{1}}X_{1i_{1}},\ldots,\sum_{i_{M}=1}^{N_{M}}X_{Mi_{M}}\right).\label{eq:Sums}
\end{align}
Since these magnetisations grow without bound as $N\rightarrow\infty$,
we need to normalise them by dividing each component $S_{\lambda}$
by a suitable power $\gamma$ of $N_{\lambda}$. The power will turn
out to depend on the regime the model is in, which we will see is
determined by $J$ and the asymptotic relative group sizes $\alpha_{1},\ldots,\alpha_{M}$:
\[
\alpha_{\lambda}:=\lim_{N\rightarrow\infty}\frac{N_{\lambda}}{N},N_{\lambda}\rightarrow\infty\text{ as }N\rightarrow\infty.
\]
Note that we assume that as the overall population goes to infinity,
so does each group. The $\alpha_{\lambda}$'s sum up to $1$. We do not
assume that all $\alpha_{\lambda}$ are necessarily positive. If $\alpha_{\lambda}=0$,
we will say that group $\lambda$ is `small'.

Throughout this article, we use Greek letters $\lambda,\mu,\nu$ to
index groups and Latin letters $i,j,k$ to index the individual spins.

This article consists of seven sections. After this introduction,
we define the three regimes of the model in Section \ref{sec:Regimes}.
Then, we present and discuss the results of this paper in Section
\ref{sec:Results}. The remainder of the article is dedicated to the
proof of these results. We introduce some combinatorial concepts in
Section \ref{sec:Combinatorial-Concepts} that are necessary for the
application of the method of moments. Next, we calculate correlations
of the spin variables of the form $\mathbb{E}\left(X_{11}\cdots X_{1k_{1}}\cdots X_{M1}\cdots X_{Mk_{M}}\right)$
in Section \ref{sec:Correlations}. The moments of $\boldsymbol{S}$
are then calculated using these correlations in Section \ref{sec:Moments}.
Finally, we prove a strong law of large numbers for the magnetisations
in Section \ref{sec:SLLN}.\\

\textsl{Acknowledgement:} We thank the referees for their careful reading of the manuscript. Their suggestions helped to improve the paper considerably.

\section{\label{sec:Regimes}The Three Regimes of the Model}

We call the regimes of the model `temperature regimes' because in
the single-group model the parameter $\beta$ can be interpreted as
the inverse temperature. There, $\beta<1$ is called the `high temperature
regime', $\beta=1$ the `critical regime', and $\beta>1$ the `low
temperature regime'. We use the same definition for homogeneous coupling
matrices, as it turns out that $\boldsymbol{S}$ behaves differently
in each of these three regimes.

Note that the Gibbs measure (\ref{eq:Pbeta}) for the homogeneous
model is identical to the Gibbs measure of the single-group model
(\ref{eq:P_single-group}). However, the  magnetisations
(\ref{eq:Sums}) form a random \emph{vector} not a scalar as in the single-group
model. Also, the entries of this random vector are correlated for
$\beta>0$. Thus, the analysis of the homogeneous model goes beyond
the classical results concerning the classical CWM. We remark that the coupling matrix $(\beta)_{\lambda\mu}$ of the homogeneous model has determinant 0, so it is not a special case but rather a
limit case of the heterogeneous model which has strictly positive determinant.
\begin{defn}
\label{def:regimes_hom}For homogeneous coupling matrices, we define
the high temperature regime to be $\beta<1$, the critical regime
to be $\beta=1$, and the low temperature regime to be $\beta>1$.
\end{defn}

For heterogeneous coupling matrices, the situation is somewhat more
complicated. The parameter space is
\begin{align*}
\Phi & :=\left\{ (\alpha_{1},\ldots,\alpha_{M})\;|\;\alpha_{1},\ldots,\alpha_{M}\geq0,\sum_{\lambda=1}^{M}\alpha_{\lambda}=1\right\} \times\left\{ J\;|\;J\text{ is an }M\times M\text{ positive definite matrix}\right\} .
\end{align*}
We define $\boldsymbol{\alpha}:=\text{diag}(\alpha_{1},\ldots,\alpha_{M})$,
where `diag' stands for a diagonal matrix with the entries given
between parentheses, and
\begin{equation}
H:=J^{-1}-\boldsymbol{\alpha}.\label{eq:def_Hessian}
\end{equation}
We call the above matrix $H$ because, as we shall later see, it is
the Hessian matrix of a function whose minima we have to find. We
define the inverse of the coupling matrix $L:=J^{-1}$. For future
reference, we also define the `square root' of the diagonal matrix
$\alpha$:
\begin{equation}
\sqrt{\boldsymbol{\alpha}}:=\text{diag}\left(\sqrt{\alpha_{1}},\ldots,\sqrt{\alpha_{M}}\right).\label{eq:sqrt_alpha}
\end{equation}

\begin{defn}
\label{def:high_temp}For heterogeneous coupling matrices, the `high
temperature regime' is the set of parameters
\[
\Phi_{h}:=\{\phi\in\Phi\;|\;H\text{ is positive definite}\}.
\]
\end{defn}

\begin{rem}
High temperature means weak interaction between voters. They influence
each other's decisions \foreignlanguage{english}{weakly}, but polarisation
of votes is still possible and happens frequently.
\end{rem}

Now we define the critical regime for heterogeneous coupling matrices.
\begin{defn}
\label{def:crit_reg}For heterogeneous coupling matrices, the `critical
regime' is the set of parameters
\[
\Phi_{c}:=\{\phi\in\Phi\;|\;H\text{ is positive semi-definite but not positive definite}\}.
\]
\end{defn}

Having defined the high temperature and critical regimes, we can now
define the low temperature regime as the complement of the union of these two sets
in the parameter space.
\begin{defn}
\label{def:low_temp}For heterogeneous coupling matrices, the `low
temperature regime' is the set of parameters
\[
\Phi_{l}:=\Phi\backslash(\Phi_{h}\cup\Phi_{c}).
\]
\end{defn}

\begin{rem}
The coupling constants in the low temperature regime are high in relation
to the inverse group sizes. Low temperature means very strong interactions
between voters.
\end{rem}

\section{\label{sec:Results}Results}

We discuss the results for homogeneous and heterogeneous coupling
matrices together, as they are qualitatively similar. We comment on
any differences.

\subsection{High Temperature Results}

Let for all $x\in\mathbb{R}^{M}$ the symbol $\delta_{x}$ stand for
the Dirac measure at the point $x$, and $\mathcal{N}(0,C)$ for the
multivariate normal distribution with mean $0$ and covariance matrix
$C$. We shall write `$\underset{N\to\infty}{\Longrightarrow}$'
for weak convergence as $N\to\infty$.

In the high temperature regime, we have a Law of Large Numbers (LLN)
and a Central Limit Theorem (CLT). First the LLN:
\begin{thm}
\label{thm:LLN_high_temp}For homogeneous and heterogeneous coupling
matrices, in their respective high temperature regimes, we have
\[
\left(\frac{S_{1}}{N_{1}},\ldots,\frac{S_{M}}{N_{M}}\right)\underset{N\to\infty}{\Longrightarrow}\delta_{0}.
\]
\end{thm}

\begin{rem}
\label{rem:LLN_high}The above LLN holds even in a strong sense, i.e.\!
in terms of almost sure convergence rather than convergence in distribution.
See Section \ref{sec:SLLN}.
\end{rem}

Next we state the CLT. In the following, let $I$ be the identity
matrix. Its dimensions should be clear from the context, such as below,
where $I$ is an $M\times M$ matrix.
\begin{thm}
\label{thm:CLT}For homogeneous and heterogeneous coupling matrices,
in their respective high temperature regimes, we have
\[
\left(\frac{S_{1}}{\sqrt{N_{1}}},\ldots,\frac{S_{M}}{\sqrt{N_{M}}}\right)\underset{N\to\infty}{\Longrightarrow}\mathcal{N}(0,C),
\]
where the covariance matrix $C$ is given by
\[
C=I+\sqrt{\boldsymbol{\alpha}}\Sigma\sqrt{\boldsymbol{\alpha}},
\]
and the matrix $\Sigma$ is $\left(\frac{\beta}{1-\beta}\right)_{\lambda,\mu=1,\ldots,M}$
if $J$ is homogeneous and $H^{-1}$ if $J$ is heterogeneous.
\end{thm}

\begin{rem}
If a group $\nu$ is small, i.e.\! $\alpha_{\nu}=0$, then the marginal
distribution of $\frac{S_{\nu}}{\sqrt{N_{\nu}}}$ asymptotically follows
a standard normal distribution and is asymptotically independent of
all other $\frac{S_{\lambda}}{\sqrt{N_{\lambda}}}$.
\end{rem}

\subsection{Critical Regime Results}

The LLN holds for the critical regime as well.
\begin{thm}
\label{thm:LLN_crit_reg}For homogeneous coupling matrices, in the
critical regime, we have
\[
\left(\frac{S_{1}}{N_{1}},\ldots,\frac{S_{M}}{N_{M}}\right)\underset{N\to\infty}{\Longrightarrow}\delta_{0}.
\]
If $M=2$ and $J$ is heterogeneous, we have
\[
\left(\frac{S_{1}}{N_{1}},\frac{S_{2}}{N_{2}}\right)\underset{N\to\infty}{\Longrightarrow}\delta_{0}.
\]
\end{thm}

\begin{rem}
\label{rem:LLN_crit}The above LLN holds even in a strong sense. See
Section \ref{sec:SLLN}.
\end{rem}

Similarly to the Central Limit Theorem \ref{thm:CLT}, we can normalise
with a power $\gamma<1$ in order to obtain a limiting distribution
which is not concentrated in the origin. In the critical regime, the
appropriate normalising power is $\gamma=3/4$.

For homogeneous coupling matrices, we have
\begin{thm}
\label{Fluctuations_hom}For homogeneous coupling matrices, in the
critical regime, we have
\[
\left(\frac{S_{1}}{N_{1}^{3/4}},\ldots,\frac{S_{M}}{N_{M}^{3/4}}\right)\underset{N\to\infty}{\Longrightarrow}\mu.
\]
The probability measure $\mu$ on $\IR^{M}$ has moments of order
$(K_{1},\ldots,K_{M}),K:=\sum_{\nu=1}^{M}K_{\nu},$ with $K$ even,
\begin{align*}
m_{K_{1},\ldots,K_{M}}(\mu) & :=12^{\frac{K}{4}}\frac{\Gamma\left(\frac{K+1}{4}\right)}{\Gamma\left(\frac{1}{4}\right)}\alpha_{1}^{\frac{K_{1}}{4}}\cdots\alpha_{M}^{\frac{K_{M}}{4}},
\end{align*}
and $m_{K_{1},\ldots,K_{M}}(\mu)=0$ if $K$ is odd.
\end{thm}

For heterogeneous coupling matrices, we have a result for two groups.
We define $L:=J^{-1}$. For $M=2$, we will call the entries
\[
L=\left(\begin{array}{cc}
L_{1} & -\bar{L}\\
-\bar{L} & L_{2}
\end{array}\right).
\]

\begin{thm}
\label{Fluctuations_het}Let $M=2$. For heterogeneous coupling matrices,
in the critical regime, we have
\[
\left(\frac{S_{1}}{N_{1}^{3/4}},\frac{S_{2}}{N_{2}^{3/4}}\right)\underset{N\to\infty}{\Longrightarrow}\mu.
\]
The probability measure $\mu$ on $\IR^{2}$ has moments of order
$(K,Q)$ $m_{K,Q}(\mu):=$
\begin{align*}
\left[\frac{12}{\alpha_{1}(L_{2}-\alpha_{2})^{2}+\alpha_{2}(L_{1}-\alpha_{1})^{2}}\right]^{\frac{K+Q}{4}}(L_{1}-\alpha_{1})^{\frac{Q}{2}}(L_{2}-\alpha_{2})^{\frac{K}{2}}\frac{\Gamma\left(\frac{K+Q+1}{4}\right)}{\Gamma\left(\frac{1}{4}\right)}\alpha_{1}^{\frac{K}{4}}\alpha_{2}^{\frac{Q}{4}}
\end{align*}
if $K+Q$ is even and $0$ otherwise.
\end{thm}

\begin{rem}
For the special case $J_{11}=J_{22}=J$, $\alpha_{1}=\alpha_{2}=1/2$,
and $J+J_{12}=2$, the moments of the limiting distribution are
\[
12^{\frac{K+Q}{4}}\frac{\Gamma\left(\frac{K+Q+1}{4}\right)}{\Gamma\left(\frac{1}{4}\right)}\alpha_{1}^{\frac{K}{4}}\alpha_{2}^{\frac{Q}{4}},
\]

and hence identical to those for the model with a homogeneous coupling
matrix and $\beta=1$.
\end{rem}

\begin{rem}
If a group $\nu$ is small, then the marginal distribution of $\frac{S_{\nu}}{N_{\nu}^{3/4}}$
is asymptotically the Dirac measure $\delta_{0}$, and therefore asymptotically
independent of all other $S_{\lambda}/N_{\lambda}^{3/4}$.
\end{rem}

It is unclear what the joint distribution $\mu$ is. However, we can
deduce the limiting distribution of two linear transformations of
$\left(\frac{S_{1}}{N_{1}^{3/4}},\frac{S_{2}}{N_{2}^{3/4}}\right)$.
Let $\nu_{\eta}$ be the probability measure on $\mathbb{R}$ given
by the density function proportional to $\exp\left(-\eta x^{4}\right),x\in\mathbb{R}$.
\begin{thm}
\label{thm:Crit_Lin_Transform}Let $M=2$ and $\alpha_{1},\alpha_{2}\neq0$.
In the critical regime, these results hold:\\
For homogeneous coupling matrices, we have
\[
\frac{\sqrt{\frac{\alpha_{2}}{\alpha_{1}}}S_{1}-\sqrt{\frac{\alpha_{1}}{\alpha_{2}}}S_{2}}{\sqrt{N}}\underset{N\to\infty}{\Longrightarrow}\mathcal{N}(0,1),
\]
and
\[
\frac{S_{1}}{2\left(\alpha_{1}N_{1}^{3}\right)^{1/4}}+\frac{S_{2}}{2\left(\alpha_{2}N_{2}^{3}\right)^{1/4}}\underset{N\to\infty}{\Longrightarrow}\nu_{\eta}
\]
with $\eta=\frac{1}{12}$.

For heterogeneous coupling matrices,
\[
\frac{\sqrt{L_{1}-\alpha_{1}}}{\sqrt{\alpha_{1}N_{1}}}S_{1}-\frac{\sqrt{L_{2}-\alpha_{2}}}{\sqrt{\alpha_{2}N_{2}}}S_{2}\underset{N\to\infty}{\Longrightarrow}\mathcal{N}\left(0,1+\frac{L_{1}-\alpha_{1}}{\alpha_{1}}+\frac{L_{2}-\alpha_{2}}{\alpha_{2}}\right),
\]
and
\[
\frac{\sqrt{L_{1}-\alpha_{1}}}{\left(\alpha_{1}N_{1}^{3}\right)^{1/4}}S_{1}+\frac{\sqrt{L_{2}-\alpha_{2}}}{\left(\alpha_{2}N_{2}^{3}\right)^{1/4}}S_{2}\underset{N\to\infty}{\Longrightarrow}\nu_{\eta}
\]
with $\eta=\frac{1}{2^{6}\cdot3}\left(\frac{\alpha_{1}}{\left(L_{1}-\alpha_{1}\right)^{2}}+\frac{\alpha_{2}}{\left(L_{2}-\alpha_{2}\right)^{2}}\right)$.
\end{thm}

The result in Theorem \ref{thm:Crit_Lin_Transform} for homogeneous
coupling matrices shows that the sequence of random variables
\[
\frac{S_{1}}{2\left(\alpha_{1}N_{1}^{3}\right)^{1/4}}+\frac{S_{2}}{2\left(\alpha_{2}N_{2}^{3}\right)^{1/4}}
\]
converges to the same limiting distribution $\nu_{1/12}$ as the sequence
$\left(S_1+S_2\right)/N^{3/4}$. The latter statement is a well-known result (see e.g.\!
Theorem V.\! 9.\! 5 in \cite{Ellis}), and Theorem \ref{thm:Crit_Lin_Transform}
says that the same limiting distribution is obtained by summing the
two suitably scaled group magnetisations $S_{1}$ and $S_{2}$.

\subsection{Low Temperature Results}

For homogeneous coupling matrices, we have a limit theorem similar
to the single-group case. As in the single-group model, we need to
solve the so called Curie-Weiss equation
\begin{align}\label{eq:CWeq}
\tanh(\beta t)=t.
\end{align}
For $\beta\leq1$, this equation has a single solution which is $t=0$.
For $\beta>1$, there are three different solutions: $-t_{1},0,t_{1}$
with $t_{1}>0$.
\begin{defn}
We define $m\left(\beta\right)$ as 0 if $\beta\leq1$ and as $t_{1}$
if $\beta>1$.
\end{defn}

\begin{thm}
\label{thm:LLN_hom_low_temp}For homogeneous coupling matrices, in
all regimes, we have
\[
\left(\frac{S_{1}}{N_{1}},\ldots,\frac{S_{M}}{N_{M}}\right)\underset{N\to\infty}{\Longrightarrow}\frac{1}{2}\left(\delta_{(-m(\beta),\ldots,-m(\beta))}+\delta_{(m(\beta),\ldots,m(\beta))}\right).
\]
\end{thm}

As far as we know, there is no limit theorem for the low temperature
regime for heterogeneous coupling matrices in the most general case.
Instead, there are results for some special cases, such as $M=2$
(see \cite{KT2}) and $M>2$ with groups of equal size (see \cite{KLSS}).
The main difficulty in the low temperature regime is the analysis
of the properties of the function $F$ defined in (\ref{eq:F}).
If the coupling matrix $J$ has strictly positive entries, then it
is clear that the global minima of $F$ are located in the positive
and negative orthant. What is not known is the number of minima of
$F$ in each orthant. If we assume that there is at most a single
local minimum in each orthant -- which holds for special cases $M=2$ and groups of equal size
mentioned above -- then we can state a limit theorem.
\begin{thm}
\label{thm:LLN_het_low_temp}For heterogeneous coupling matrices with
positive entries, assume the function $F$ defined in (\ref{eq:F})
has at most one local minimum in each orthant, and let $\bar{m}\in\IR^{M}$
be the local minimum found in the positive orthant. We set $m:=\tanh{\bar{m}}$, applying the $\tanh$ function componentwise. Then, in the low
temperature regime, we have
\[
\left(\frac{S_{1}}{N_{1}},\ldots,\frac{S_{M}}{N_{M}}\right)\underset{N\to\infty}{\Longrightarrow}\frac{1}{2}\left(\delta_{-m}+\delta_{m}\right).
\]
\end{thm}
Note that above $m$ is a point in $\IR^M$, as opposed to the previous theorem concerning homogeneous coupling matrices.

Other limit theorems of the same type hold for different assumptions
on the signs of the non-diagonal entries of $J$.

The previous theorem says that in the low temperature regime the magnetisations are significant in the sense that the Law of Large Numbers as in Theorems \ref{thm:LLN_high_temp} and \ref{thm:LLN_crit_reg} does not hold here. We can investigate the fluctuations around the two points of concentration $\pm m$ which turn out to be normally distributed under suitable scaling. This is the subject of the next theorem, which is a conditional CLT for the low temperature regime.

\begin{thm}
\label{thm:cond_CLT}Assume one of the following assumptions hold:
\begin{enumerate}
\item Let $J$ be a homogeneous coupling matrix.
\item Let $J$ be a heterogeneous coupling matrix with positive entries
and let $M=2$.
\item Let $J$ be a heterogeneous coupling matrix with positive entries,
and assume the function $F$ defined in \ref{eq:F} has at most one local
minimum in each orthant of $\IR^{M}$.
\end{enumerate}
We set $m:=\left(m\left(\beta\right),\ldots,m\left(\beta\right)\right)$
if $J$ is homogeneous and $m:=\tanh\bar{m}$ if $J$ is heterogeneous
and $\bar{m}$ the minimum of $F$ in the positive orthant.

Then, conditioning on $S_{\nu}>0$ for all groups $\nu$, we have
\[
\left(\frac{1}{\sqrt{N_{1}}}\sum_{i_{1}=1}^{N_{1}}\left(X_{1i_{1}}-m_{1}\right),\ldots,\frac{1}{\sqrt{N_{M}}}\sum_{i_{M}=1}^{N_{M}}\left(X_{Mi_{M}}-m_{M}\right)\right)\underset{N\to\infty}{\Longrightarrow}\mathcal{N}(0,E).
\]
Similarly, under the condition $S_{\nu}<0$ for all groups $\nu$,
we have
\[
\left(\frac{1}{\sqrt{N_{1}}}\sum_{i_{1}=1}^{N_{1}}\left(X_{1i_{1}}+m_{1}\right),\ldots,\frac{1}{\sqrt{N_{M}}}\sum_{i_{M}=1}^{N_{M}}\left(X_{Mi_{M}}+m_{M}\right)\right)\underset{N\to\infty}{\Longrightarrow}\mathcal{N}(0,E).
\]
The covariance matrix $E$ is the same in both limiting distributions, and
\[
E=\textup{diag}\left(1-m_{\lambda}^{2}\right)+\sqrt{\boldsymbol{\alpha}}H^{-1}\left(\bar{m}\right)\sqrt{\boldsymbol{\alpha}},
\]
where $H^{-1}\left(\bar{m}\right)$ is the inverse of the Hessian
matrix of $F$ at $\pm\bar{m}$.
\end{thm}

In the remainder of this paper, we will prove the above results by
the method of moments. The structure of most of these proofs is the
following:
\begin{enumerate}
\item By expanding the moments
\begin{align*}
   \IE\left[\left(\sum_{i_{1}=1}^{N_{1}}X_{1 i_{1}}\right)^{K_{1}}\left(\sum_{i_{2}=1}^{N_{2}}X_{2 i_{2}}\right)^{K_{2}}\cdots
   \left(\sum_{i_{M}=1}^{N_{M}}X_{Mj_{M}}\right)^{K_{M}}\right],
\end{align*}
we obtain a huge sum with correlations of the form
\begin{align*}
   \mathbb{E}\left(\prod_{\nu=1}^{M}\;\prod_{k_{\nu}=1}^{K_{\nu}} X_{\nu\, i_{\nu k_{\nu}}}\right)~=~  \mathbb{E}\left(X_{1\,i_{11}}\cdots X_{1i_{1\,K_{1}}} \cdots X_{M\,i_{M1}}\cdots X_{M\,i_{MK_{M}}}\right)\,.\label{eq:moments}
\end{align*}
\item We calculate these correlations asymptotically
for large $N$:
\begin{enumerate}
\item First we use a Hubbard-Stratonovich transformation to express probabilities
given by the Curie-Weiss measure $\mathbb{P}$ as an integral: for
the single-group model, this transformation consists of
\[
\exp\left(\frac{\beta}{2N}(\sum_{i=1}^{N}X_{i})^{2}\right)=\frac{1}{\sqrt{2\pi}}\int_{\mathbb{R}}\exp\left(-\frac{y^{2}}{2}\right)\exp\left(y\sqrt{\frac{\beta}{N}}\sum_{i=1}^{N}X_{i}\right)\text{d}y.
\]
\item Then we employ Laplace's method to estimate the above integral for large $N$.
\end{enumerate}
\item Finally, we calculate the moments of the normalised $\boldsymbol{S}$, and prove they converge to the claimed limit.
\end{enumerate}

\section{\label{sec:Combinatorial-Concepts}Combinatorial Concepts}

Whenever we use the method of moments, we will have to evaluate sums
of the form
\begin{align*}
&\IE\left[\left(\sum_{i_{1}=1}^{N_{1}}X_{1 i_{1}}\right)^{K_{1}}\left(\sum_{i_{2}=1}^{N_{2}}X_{2 i_{2}}\right)^{K_{2}}\cdots
   \left(\sum_{i_{M}=1}^{N_{M}}X_{Mj_{M}}\right)^{K_{M}}\right]\\
   =~&\sum_{i_{11},\ldots,i_{1K_{1}}}\;\sum_{i_{21},\ldots,i_{2K_{2}}}\;\ldots\; \sum_{i_{M1},\ldots,i_{MK_{M}}}\;\mathbb{E}\left(\prod_{\nu=1}^{M}\;\prod_{k_{\nu}=1}^{K_{\nu}} X_{\nu\, i_{\nu k_{\nu}}}\right).
\end{align*}
To do the book-keeping for these huge sums we introduce a few combinatorial
concepts taken from \cite{MM}. Let $|A|$ stand for the cardinality
of the set $A$.
\begin{defn}
Let $L$ be a natural number. We define a multiindex $\underline{i}=(i_{1},i_{2},\ldots,i_{L})\in\{1,2,\ldots,N\}^{L}$.
\begin{enumerate}
\item For $j\in\{1,2,\ldots,N\}$, we set
\[
\nu_{j}(\underline{i}):=|\{k\in\{1,2,\ldots,L\}\,|\,i_{k}=j\}|.
\]
\item For $\ell=0,1,\ldots,L$, we define
\[
\rho_{\ell}(\underline{i}):=|\{j\,|\,\nu_{j}(\underline{i})=\ell\}|
\]
and
\[
\underline{\rho}(\underline{i}):=(\rho_{1}(\underline{i}),\ldots,\rho_{L}(\underline{i})).
\]
\end{enumerate}
\end{defn}

The numbers $\nu_{j}(\underline{i})$ represent the multiplicity of
each index $j\in\{1,2,\ldots,N\}$ in the multiindex $\underline{i}$,
and $\rho_{\ell}(\underline{i})$ represents the number of indices
in $\underline{i}$ that occur exactly $\ell$ times. We shall call
$\underline{\rho}(\underline{i})$ the profile of the multiindex $\underline{i}$.
\begin{lem}
\label{lem:sum-l-rho}For all $\underline{i}=(i_{1},i_{2},\ldots,i_{L})\in\{1,2,\ldots,N\}^{L}$,
we have $\sum_{\ell=1}^{L}\ell\rho_{\ell}(\underline{i})=L$.
\end{lem}

We use this basic property of profiles to define
\begin{defn}
\label{def:profile_vector}Let $\underline{r}=(r_{1},\ldots,r_{L})$
be such that $\sum_{\ell=1}^{L}\ell r_{\ell}=L$ hold. We call $\underline{r}$
a profile vector. We define
\[
w_{L}(\underline{r})=\left|\{\underline{i}\in\{1\ldots,N\}^{L}\,|\,\underline{\rho}(\underline{i})=\underline{r}\}\right|
\]
to represent the number of multiindices $\underline{i}$ that have
a given profile vector $\underline{r}$.
\end{defn}

We now define the set of all profile vectors for a given $L\in\mathbb{N}$.
\begin{defn}
\label{def:profile_sets}Let $\Pi^{(L)}=\left\{ \underline{r}\in\{0,1,\ldots,L\}^{L}\,|\,\sum_{\ell=1}^{L}\ell r_{\ell}=L\right\} $.
Some important subsets of $\Pi^{(L)}$ are $\Pi_{k}^{(L)}=\left\{ \underline{r}\in\Pi^{(L)}\,|\,r_{1}=k\right\} $,
$\Pi^{0(L)}=\left\{ \underline{r}\in\Pi^{(L)}\,|\,r_{\ell}=0\text{ for all }\ell\geq3\right\} $,
and $\Pi^{+(L)}=\left\{ \underline{r}\in\Pi^{(L)}\,|\,r_{\ell}>0\text{ for some }\ell\geq3\right\} $.
We can also combine superscripts and subscripts. Then we have, e.g.,
$\Pi_{0}^{0(L)}=\left\{ \underline{r}\in\Pi^{(L)}\,|\,r_{\ell}=0\text{ for all }\ell\neq2\right\} $.
\end{defn}

We shall write for any $\underline{i}\in\{1,2,\ldots,N\}^{L}$ $X_{\underline{i}}=X_{i_{1}}\cdots X_{i_{L}}$.
For any $\underline{r}\in\Pi^{(L)}$, let $\underline{j}\in\{1,2,\ldots,N\}^{L}$
be such that $\underline{\rho}\left(\underline{j}\right)=\underline{r}$.
Then we let $X_{\underline{r}}$ stand for $X_{\underline{j}}$. This
definition is not problematic if we are only interested in the expectation
\[
\mathbb{E}\left(X_{\underline{r}}\right)=\mathbb{E}\left(X_{\underline{j}}\right),
\]
and the random variables $X_{1},\ldots,X_{N}$ are exchangeable.

If there are $M$ sets $\{1,2,\ldots,N_{\nu}\}^{L_{\nu}}$, and for
each $\nu$ $\underline{i_{\nu}}\in\{1,2,\ldots,N_{\nu}\}^{L_{\nu}}$,
then we set $\text{\ensuremath{\underbar{i}}\,:=\ensuremath{\left(\underline{i_{\nu}}\right)}}$
and write $X_{\underbar{i}}$ for $X_{\underline{i_{1}}}\cdots X_{\underline{i_{M}}}$.
Similarly, if we have profile vectors $\underline{r_{\nu}}\in\Pi^{(L_{\nu})}$,
and $\text{\ensuremath{\underline{j_{\nu}}\in}\{1,2,\ensuremath{\ldots},\ensuremath{N_{\nu}\}^{L_{\nu}}}}$
such that $\underline{\rho}\left(\underline{j_{\nu}}\right)=\underline{r_{\nu}}$,
then we write $X_{\underbar{r}}$ for $X_{\underline{j_{1}}}\cdots X_{\underline{j_{M}}}$.
\begin{prop}
\label{thm:comb-coeff-multiindex}For $\underline{r}\in\Pi^{(L)}$,
set $r_{0}:=N-\sum_{\ell=1}^{L}r_{\ell}$. Then
\[
w_{L}(\underline{r})=\frac{N!}{r_{1}!r_{2}!\ldots r_{L}!r_{0}!}\frac{L!}{1!^{r_{1}}2!^{r_{2}}\cdots L!^{r_{L}}}.
\]
If we let $N$ go to infinity, then we have
\[
w_{L}(\underline{r})\approx\frac{N^{\sum_{l=1}^{L}r_{l}}}{r_{1}!r_{2}!\ldots r_{L}!}\frac{L!}{1!^{r_{1}}2!^{r_{2}}\cdots L!^{r_{L}}}.
\]
\end{prop}

This proposition is based on Theorem 3.14 and Corollary 3.18 in \cite{MM}.

\section{\label{sec:Correlations} Correlations}

In this section, we compute correlations of the form
\begin{align}
   \mathbb{E}\left(\prod_{\nu=1}^{M}\;\prod_{k_{\nu}=1}^{K_{\nu}} X_{\nu\,i_{\nu\,k_{\nu}}}\right)~=~  \mathbb{E}\left(X_{1\,i_{11}}\cdots X_{1i_{1\,K_{1}}} \cdots X_{M\,i_{M1}}\cdots X_{M\,i_{MK_{M}}}\right)\,,\label{eq:moments}
\end{align}
where for $\nu=1,\ldots,M $ the sequence $i_{\nu k_{\nu}}$ addresses $K_{\nu} $ spins in the group $\nu $.

Since $X_{\nu i}^{2}=1 $, it is enough to compute \eqref{eq:moments} for pairwise distinct index pairs $(\nu,i_{\nu j})_{\nu=1,\ldots,M\;j=1,\ldots,K_{\nu}} $.

From the definition in \eqref{def:model}, it is clear that the $X_{\nu i}$ are exchangeable \emph{within groups}, i.e.\! for
pairwise distinct index pairs, the random vectors
\begin{align}\notag
   \left(X_{1\,i_{11}},\ldots, X_{1i_{1\,K_{1}}} ,\;\ldots\;, X_{M\,i_{M1}},\ldots, X_{M\,i_{MK_{M}}}\right)
   \quad\text{and}\quad
   \left(X_{1\,1}, \ldots, X_{1\,K_{1}} ,\;\ldots\;, X_{M\,1},\ldots, X_{M\,K_{M}}\right)
\end{align}
have the same distributions. Consequently, for these indices
\begin{align}\label{eq:correlation}
   m_{K_{1},\ldots,K_{M}}~:=~\mathbb{E}\left(\prod_{\nu=1}^{M}\;\prod_{k_{\nu}=1}^{K_{\nu}} X_{\nu\,i_{\nu\,k_{\nu}}}\right)
   ~=~\IE\left(\prod_{\nu=1}^{M}\;\prod_{i_{\nu}=1}^{K_{\nu}} X_{\nu\,i_\nu}\right).
\end{align}
So the identity of the specific indices chosen from each group does not affect the correlation; only the number of different indices matters.
The correlations differ according to the class of coupling matrix
and the regime. We first show how to calculate the correlation for
homogeneous coupling. These results have been known for a while and
can also be used to calculate moments in the single-group model.

\subsection{Homogeneous Coupling}

When there is homogeneous coupling, \emph{all} random variables are exchangeable and the (joint) distribution of the $X_{\nu i}$ is actually the same as the distribution of a single-group CWM with $N=\sum_\nu N_{\nu}$.

To calculate the correlations \eqref{eq:correlation} in this case, we may therefore use estimates known for the single-group model.

\begin{defn}\label{def:appr}
 Real-valued sequences $f_{N},g_{N}$ are called \emph{asymptotically equal} (as $N\to\infty$), in short $f_{N}\approx g_{N}$, if
 \begin{align*}
    \lim_{N\rightarrow\infty}\frac{f_{N}}{g_{N}}~=~1.
 \end{align*}
\end{defn}
The next theorem gives asymptotic expressions for the expectations.
\begin{thm}
\label{thm:EXY-hom}For homogeneous coupling matrices, the expectations
\begin{align*}
   m_{k_{1},\ldots,k_{M}}~=~\mathbb{E}\left(X_{11}\cdots X_{1k_{1}}\cdots X_{M1}\cdots X_{Mk_{M}}\right)
\end{align*}
are asymptotically equal to:
\begin{enumerate}
\item if $\beta<1$,
\[
(k_{1}+\cdots+k_{M}-1)!!\left(\frac{\beta}{1-\beta}\right)^{\frac{k_{1}+\cdots+k_{M}}{2}}\frac{1}{N^{\frac{k_{1}+\cdots+k_{M}}{2}}};
\]
\item if $\beta=1$,
\[
12^{\frac{k_{1}+\cdots+k_{M}}{2}}\frac{\Gamma\left(\frac{k_{1}+\cdots+k_{M}+1}{4}\right)}{\Gamma\left(\frac{1}{4}\right)}\frac{1}{N^{\frac{k_{1}+\cdots+k_{M}}{4}}};
\]
\item if $\beta>1$,
\[
m(\beta)^{k_{1}+\cdots+k_{M}}\,,
\]
where $m(\beta)$ is the unique strictly positive solution of \eqref{eq:CWeq};
\end{enumerate}
\emph{provided} $k_{1}+\cdots+k_{M}$ is even.
If $k_{1}+\cdots+k_{M}$ is odd, then $m_{k_{1},\ldots,k_{M}}=0$
for all $\beta\geq0$.
\end{thm}

\begin{proof}
This is Theorem 4.3 in \cite{WK-M}.
\end{proof}

\subsection{Heterogeneous Coupling}

When there is heterogeneous coupling, only the random variables within
each group $X_{\lambda i}$ are exchangeable. The case $M=2$ and
high temperature was analysed in a similar fashion in \cite{KT2}.

In the following, we set
\begin{align*}
\mathbf{x}:=(x_{11},\ldots,x_{1N_{1}}, \ldots,x_{MN_{M}})\in\{ -1,1 \}^{N}, \quad \quad
s_{\lambda}(\mathbf{x})~:=~\sum_{i=1}^{N_{\lambda}}x_{\lambda i},\quad \text{and}\quad s(\bx):=\left(s_{1},\ldots,s_{M}\right)^{T}.
\end{align*}
We calculate the correlations by adapting Laplace's method. The assumption
that the coupling matrix is positive definite is of key importance for the Hubbard-Stratonovich
transformation:
\begin{prop}
[Multidimensional Hubbard-Stratonovich]Let $J$ be a heterogeneous
coupling matrix, and\\ $\bx=\left(x_{11},\ldots,x_{MN_{M}}\right)\in\{ -1,1 \}^{N}$ a spin
configuration.  Then
\begin{align}\label{eq:HS}
e^{-\mathbb{H}(\bx)} & =e^{\frac{s(\bx)^{T}Js(\bx)}{2N}}=
\frac{\sqrt{\det J}}{(2\pi N)^{\frac{M}{2}}}\int_{\mathbb{R}^{M}}e^{-\frac{u^{T}Ju}{2N}}e^{\frac{u^{T}J\,s(\bx)}{N}}\mathrm{d}u.
\end{align}
\end{prop}

\begin{proof}
This is a straightforward calculation.
\end{proof}

We change variables $y:=N^{-1}Ju$ in \eqref{eq:HS}. Recall that we write $L$ for
$J^{-1}$, and we obtain
\begin{align}\label{eq:HS2}
e^{-\mathbb{H}(\bx)} =\frac{N^{\frac{M}{2}}}{(2\pi)^{\frac{M}{2}}\sqrt{\det J}}\int_{\mathbb{R}^{M}}e^{-\frac{N}{2}\,(y^{T}L\,y)}\,e^{y^{T}s(\bx)}\,\mathrm{d}y.
\end{align}
Let us define
\begin{align*}
   \mathcal{Z}_{k_{1},\ldots,k_{M}}~
   :=~&\sum_{\bx\in\{ -1,1 \}^{N}}~\left[\prod_{i_{1}=1}^{k_{1}}x_{1i_{1}}\right]
   \cdots\left[\prod_{i_{M}=1}^{k_{M}}x_{Mi_{M}}\right]~e^{-\mathbb{H}(\bx)}.
\end{align*}
Observe that
\begin{align*}
  \frac{\mathcal{Z}_{k_{1},\ldots,k_{M}}}{\mathcal{Z}_{0,\ldots,0}}~=~\mathbb{E}\left(X_{11}\cdots X_{1k_{1}}\cdots X_{M1}\cdots X_{Mk_{M}}\right)
   ~=~m_{k_{1},\ldots,k_{M}}.
\end{align*}
Using \eqref{eq:HS2} and letting $c_N$ stand for the factor multiplying the integral in \eqref{eq:HS2}, we compute
\begin{align*}
   \mathcal{Z}_{k_{1},\ldots,k_{M}}~&=~c_{N}\;\int_{\IR^{M}} e^{-\frac{N}{2}y^{T}L\,y}\;\sum_{\bx\in\{ -1,1 \}^{N}}
   \left[\prod_{\lambda=1}^{M}\left(\prod_{i_{\lambda}=1}^{k_{\lambda}}x_{\lambda\, i_{\lambda}}\right)\right]\,e^{-y^{T}s(\bx)}\,\text{d}y\notag\\
   &=~c_{N}\;\int_{\IR^{M}} e^{\frac{N}{2}y^{T}L\,y}\;\prod_{\lambda=1}^{M}\left[\left(\sum_{x\in\{ -1,1 \}}x\,e^{xy_{\lambda}}\right)^{k_{\lambda}}\,\left(\sum_{x\in\{ -1,1 \}}e^{xy_{\lambda}}\right)^{N_{\lambda}-k_{\lambda}}\right]\;\text{d}y\notag\\
   &=~2^{N}\,c_{N}\;\int_{\IR^{M}}\;\left(\prod_{\lambda=1}^{M}\,\tanh(y_{\lambda})^{k_{\lambda}}\right)\; e^{-\left(\frac{N}{2}y^{T}L\,y-\sum_{\lambda=1}^{M}N_{\lambda}\,\ln\left(\cosh(y_{\lambda})\right)\right)}\;\text{d}y.
\end{align*}

We summarise:
\begin{thm}\label{thm:Finetti}
   Define the functions $F, F_{N}:\IR^{M}\to\IR$ by:
   \begin{align}\label{eq:F}
      F_{N}(y)~&:=~\frac{1}{2}\sum_{\lambda,\nu=1}^{M} J^{-1}_{\lambda \nu}\,y_{\lambda}\,y_{\nu}\;-\;\sum_{\lambda=1}^{M} \;\frac{N_{\lambda}}{N}\ln\,\cosh(y_{\lambda}),\quad
      F(y)~:=~\frac{1}{2}\sum_{\lambda,\nu=1}^{M} J^{-1}_{\lambda \nu}\,y_{\lambda}\,y_{\nu}\;-\;\sum_{\lambda=1}^{M} \;\alpha_{\lambda}\ln\,\cosh(y_{\lambda})\,.
   \end{align}
   Then
   \begin{align}\label{eq:Finetti0}
      m_{k_{1},\ldots,k_{M}}~&=~\mathbb{E}\left(X_{11}\cdots X_{1k_{1}}\cdots X_{M1}\cdots X_{Mk_{M}}\right)\notag\\
      ~&=~      \frac{\displaystyle \int_{\IR^{M}}\;\left(\prod_{\lambda=1}^{M}\,\tanh(y_{\lambda})^{k_{\lambda}}\right)\; e^{-N\,F_{N}(y)}\;\textup{d}y}
      {\displaystyle \int_{\IR^{M}}\; e^{-N\,F_{N}(y)}\;\textup{d}y}
      ~\approx~\frac{\displaystyle \int_{\IR^{M}}\;\left(\prod_{\lambda=1}^{M}\,\tanh(y_{\lambda})^{k_{\lambda}}\right)\; e^{-N\,F(y)}\;\textup{d}y}
      {\displaystyle \int_{\IR^{M}}\; e^{-N\,F(y)}\;\textup{d}y}.
   \end{align}
\end{thm}

Let us define the probability measure $P_{y}$ on $\{ -1,1 \} $ by $P_{y}(1)=\frac{1}{2}(1+\tanh y)$ and its $k$-fold product by $P_{y}^{\otimes k}$. The corresponding expectation is denoted by $E_{y}$ and $E_{y}^{\otimes k}$, respectively.

From Theorem \ref{thm:Finetti}, we easily get a de-Finetti-type representation for the Curie-Weiss measure $\IP$ defined in
\eqref{eq:Pbeta}:

\begin{cor}\label{cor:Finetti}
We have
   \begin{align}
      &\mathbb{P}\left(X_{11}=x_{11}\ldots X_{1N_{1}}=x_{1N_{1}}\ldots X_{M1}\ldots X_{MN_{M}}=x_{MN_{M}}\right)\notag\\
      ~&=~
      \frac{1}{\int_{\IR^{M}}\; e^{-N\,F_{N}(y)}\;\textup{d}y}\;\, \int_{\IR^{M}}\;P_{y_{1}}^{\otimes N_{1}}(x_{11},\ldots,x_{1N_{1}})\cdots P_{y_{M}}^{\otimes N_{M}}(x_{M1},\ldots,x_{MN_{M}})\; e^{-N\,F_{N}(y)}
      \;\textup{d}y\label{eq:Finetti1}\\
      &\approx~\frac{1}{\int_{\IR^{M}}\; e^{-N\,F(y)}\;\textup{d}y}\;\, \int_{\IR^{M}}\;P_{y_{1}}^{\otimes N_{1}}(x_{11},\ldots,x_{1N_{1}})\cdots P_{y_{M}}^{\otimes N_{M}}(x_{M1},\ldots,x_{MN_{M}})\; e^{-N\,F(y)}
      \;\textup{d}y \nonumber
      \intertext{and}
      & m_{k_{1},\ldots,k_{M}}~=~\frac{1}{\int_{\IR^{M}}\; e^{-N\,F_{N}(y)}\;\textup{d}y}\;\, \int_{\IR^{M}}\;\left[E_{y_{1}}^{\otimes N_{1}}\left(\prod_{i=1}^{k_{1}}x_{1i}\right)\right]\cdots \left[ E_{y_{M}}^{\otimes N_{M}}\left(\prod_{i=1}^{k_{M}} x_{Mi}\right)\right] \; e^{-N\,F_{N}(y)}
      \;\textup{d}y\label{eq:Finetti2}\\
      ~&\approx~\frac{1}{\int_{\IR^{M}}\; e^{-N\,F(y)}\;\textup{d}y}\;\, \int_{\IR^{M}}\;\left[E_{y_{1}}^{\otimes N_{1}}\left(\prod_{i=1}^{k_{1}}x_{1i}\right)\right]\cdots \left[ E_{y_{M}}^{\otimes N_{M}}\left(\prod_{i=1}^{k_{M}} x_{Mi}\right)\right] \; e^{-N\,F(y)}
      \;\textup{d}y. \nonumber
   \end{align}
\end{cor}
\begin{proof}
Equation \eqref{eq:Finetti2} is just a rewriting of \eqref{eq:Finetti0}. Equation \eqref{eq:Finetti1} follows since the probability measures involved are determined by their moments.
\end{proof}

Our goal is to apply Laplace's method to evaluate the integrals in \eqref{eq:Finetti0} asymptotically. In order to do so, we have to analyse the critical points of the function $F$. We start our analysis with the high temperature regime.

\subsubsection{High Temperature Regime}
\begin{prop}
\label{prop:pos-def}In the high temperature regime, i.e.\! if the matrix $H=J^{-1}-\boldsymbol{\alpha}=L-\boldsymbol{\alpha}$
is positive definite, the function $F$ defined in \eqref{eq:F}
 has a unique
global minimum at the origin.  The matrix $H$ is the Hessian of $F$ at $0$.
\end{prop}

\begin{proof}
The Hessian of $F$ is given by
\begin{align*}
   H(y)~:=~\left(\frac{\partial^{2}F}{\partial x_{\lambda} \partial x_{\nu}}(y)\right)~=~J^{-1}_{\;\lambda\nu}-\delta_{\lambda\nu}\;\alpha_{\nu}
   \left(1-\tanh\left(y_{\nu}\right)^{2}\right),
\end{align*}
where the symbol $\delta_{\lambda\nu}$ stands for the Kronecker delta.
Since $H=H(0)$ is assumed to be positive definite, by monotonicity of the eigenvalues of $H(y)$ these matrices are positive definite for all $y$. Thus, the function $F$ is strictly convex.

The gradient of $F$ vanishes in the origin. Hence, the origin is a local minimum. By strict convexity, $0$ is the unique minimum.

\end{proof}

\begin{rem}
   The above proof actually shows that the origin is the unique minimum of $F$ even if $H$ is merely positive \emph{semi-definite}. In fact, for positive semi-definite $H$, the Hessian $H(y)$ of $F$ is positive definite for all $y\not=0$.
\end{rem}
\begin{defn}
   If $A$ is a positive definite $M\times M$-matrix and $\mathbf{k}\in \mathbb{N}_{0}^{M}$ is a multiindex,
   we denote by $m_{\mathbf{k}}(A)$ the $\mathbf{k}$-th moment of the multivariate normal distribution $\mathcal{N}(0,A)$, i.e.
   \begin{align*}
      m_{\mathbf{k}}(A)~:=~\frac{1}{\left(2\pi\det A\right)^{M/2}}\,\int_{\IR^{M}}\;x_{1}^{\,k_{1}}\cdots x_{M}^{\,k_{M}}\;e^{-\frac{1}{2}\,(x^{T}A^{-1}x)}\;\text{d}x.
   \end{align*}
\end{defn}
\begin{thm}
\label{thm:EXY-het_high} Let $X_{11},\ldots,X_{1N_{1}},\ldots, X_{M1},\ldots,X_{MN_{M}}$ be a sequence of random variables which are M-group Curie-Weiss distributed with positive definite coupling matrix $J $ and assume that $\lim_{N\to\infty}\frac{N_{\nu}}{N}=\alpha_{\nu}$.

If the model is in the high temperature regime, i.e.\! if the matrix $H=J^{-1}-\boldsymbol{\alpha}$ is positive definite, then
\begin{align*}
   m_{k_{1},\ldots,k_{M}}~:=~\IE\left(X_{11}\cdots X_{1k_{1}}\cdots X_{M1}\cdots X_{Mk_{M}}\right)~\approx~\frac{1}{N^{\frac{k_{1}+\cdots +k_{M}}{2}}}\;
   m_{\mathbf{k}}(H^{-1})\quad \textup{as } N\to\infty.
\end{align*}

\end{thm}

Theorem \ref{thm:EXY-het_high} follows from the following proposition which can be considered a special case of Laplace's method in higher dimension (for background on Laplace's method see for example \cite{wong}).

\begin{prop}\label{prop:Laplace_multivariate}
Suppose $H=J^{-1}-\boldsymbol{\alpha}$ is positive definite. For $\mathbf{k}=(k_{1},\ldots,k_{M})\in \mathbb{N}_{0}^{M}$, $k=\sum k_{\nu}$, and for $F$ as in \eqref{eq:F}, we have
\begin{align*}
   \int_{\IR^{M}}\;\left(\prod_{\lambda=1}^{M}\,\tanh(y_{\lambda})^{k_{\lambda}}\right)\; e^{-N\,F(y)}\;\text{d}y~\approx~
   \frac{1}{N^{(M+k)/2}}\;\int_{\IR^{M}}\,\left(\prod_{\nu=1}^{M} {y_{\nu}}^{k_{\nu}}\right)\;e^{-\frac{1}{2}\,(y^{T}H\,y)}\;\textup{d}y.
\end{align*}
\end{prop}
\begin{proof}[Proof (Theorem \ref{thm:EXY-het_high})]
We assume Proposition \ref{prop:Laplace_multivariate} for the moment. Recall that
\begin{align*}
  m_{k_{1},\ldots,k_{M}}~=~\frac{\mathcal{Z}_{k_{1},\ldots,k_{M}}}{\mathcal{Z}_{0,\ldots,0}}~\approx~
  \frac{\displaystyle\int_{\IR^{M}}\;\left(\prod_{\lambda=1}^{M}\,\tanh(y_{\lambda})^{k_{\lambda}}\right)\; e^{-N\,F(y)}\;\text{d}y}{\displaystyle\int_{\IR^{M}}\; e^{-N\,F(y)}\;\text{d}y}.
\end{align*}
Applying Proposition \ref{prop:Laplace_multivariate} to both the numerator and the denominator in the above expression
and noting that
\begin{align*}
   \int_{\IR^{M}}\;e^{-\frac{1}{2}\,(y^{T}H\,y)}\;\text{d}y~=~\left(2\pi\det (H^{-1})\right)^{M/2}
\end{align*}
 proves the theorem.
\end{proof}
We next prove Proposition \ref{prop:Laplace_multivariate}.

\begin{proof}[Proof (Proposition \ref{prop:Laplace_multivariate})]
Let us use the shorthand notation
\begin{align*}
   \tanh^{\mathbf{k}}({y})~=~\prod_{\lambda=1}^{M}\,\tanh(y_{\lambda})^{k_{\lambda}}\quad\text{and}\quad
   {y}^{\mathbf{k}}~=~\prod_{\nu=1}^{M} y_{\nu}^{k_{\nu}}
\end{align*}
in this proof as well as the remainder of this article.

We compute, by a change of variables,
\begin{align}\label{eq:cov}
   &N^{\frac{k+M}{2}}\,\int_{\IR^{M}}\,\tanh^{\mathbf{k}}({y})~e^{-N\left(\eh y^{T}J^{-1}y\,-\sum \alpha_{\nu}\ln\cosh(y_{\nu})\right)}\dd{y}\notag\\
   =~&N^{\frac{k}{2}}\,\int_{\IR^{M}}\,\tanh^{\mathbf{k}}\left(\frac{x}{\sqrt{N}}\right)\,\;e^{-\left(\eh x^{T}J^{-1}x\,-\sum \alpha_{\nu}\,N\ln\cosh\left(\frac{x_{\nu}}{\sqrt{N}}\right)\right)}\dd{x}\,.
\end{align}
Now, we expand $\ln\cosh(x)$ and $\tanh(x)$ around $x_0=0$ using Taylor's theorem (in Lagrange form):
\begin{align*}
   \ln\cosh(x)~=~\eh\, x^{2}\,-\eh \tanh^{2}(\bar{x})\,x^{2}\quad \textup{and} \quad \tanh(x)~=~x-\frac{\tanh(\tilde{x})}{\cosh^{2}(\tilde{x})}\,x^{2}\,,
\end{align*}
where $\bar{x}$ and $\tilde{x}$ are points between $0$ and $x$.

Inserting this in \eqref{eq:cov} gives
\begin{align}\label{eq:Taylor1}
   \int_{\IR^{M}}\,\prod_{\nu=1}^{M} \left(x_{\nu}-\,\frac{1}{\sqrt{N}}\frac{\tanh(\tilde{x})}{\cosh^{2}(\tilde{x})}\,x_{\nu}^{\,2}\right)^{k_{\nu}}        e^{-\eh\left(y^{T}(J^{-1}-\boldsymbol{\alpha})\,y)\,-\,\frac{1}{2N}\sum \alpha_{\nu}\,\tanh^{2}(\bar{x})\,x^{2}\right)}\dd{x}\,.
\end{align}
The integrand in \eqref{eq:Taylor1} converges to $x^{\bk}e^{-\eh(x^{T} H x)}$. Moreover, the integrand is bounded in absolute value by an expression of the form $C(1+x^{2})^{k}e^{-\gamma x^{2}}$ which is integrable. We may therefore apply the dominated convergence theorem and arrive at the assertion of the proposition.
\end{proof}

\subsubsection{Critical Regime}

Now we turn to the critical regime. We already know from Proposition
\ref{prop:pos-def} that $F$ has a unique global minimum at the origin
in the critical regime as well as in the high temperature regime.
For the proof of this result, see \cite{KT2}.
We remind the reader that we set
\begin{align*}
   L~:=~J^{-1}~=~\left(\begin{array}{cc}
L_{1} & -\bar{L}\\
-\bar{L} & L_{2}
\end{array}\right).
\end{align*}
\begin{thm}
\label{thm:EXY-het_crit}Let $L_{\nu}-\alpha_{\nu}>0$ for both groups
and $(L_{1}-\alpha_{1})(L_{2}-\alpha_{2})=\bar{L}^{2}$. Then, for
all $K,Q\in\mathbb{N}_{0}$, the expected value $\mathbb{E}(X_{11}\cdots X_{1K}X_{21}\cdots X_{2Q})$
is asymptotically given by the expression
\[
\left[\frac{12}{\alpha_{1}(L_{2}-\alpha_{2})^{2}+\alpha_{2}(L_{1}-\alpha_{1})^{2}}\right]^{\frac{K+Q}{4}}(L_{1}-\alpha_{1})^{\frac{Q}{2}}(L_{2}-\alpha_{2})^{\frac{K}{2}}\frac{\Gamma\left(\frac{K+Q+1}{4}\right)}{\Gamma\left(\frac{1}{4}\right)}\frac{1}{N^{\frac{K+Q}{4}}}
\]
if $K+Q$ is even and $0$ otherwise.
\end{thm}

\begin{rem}
For the special case $M=2$,
\[
J=\left(\begin{array}{cc}
J_{1} & \bar{J}\\
\bar{J} & J_{2}
\end{array}\right),
\]
 $J_{1}=J_{2}=J$, $\alpha_{1}=\alpha_{2}=1/2$, and $J+\bar{J}=2$,
the correlations are asymptotically equal to
\[
12^{\frac{K+L}{4}}\frac{\Gamma\left(\frac{K+L+1}{4}\right)}{\Gamma\left(\frac{1}{4}\right)}\frac{1}{N^{\frac{K+L}{4}}}.
\]

These correlations are identical to those for the model with homogeneous
coupling matrix and $\beta=1$.

\subsubsection{Low Temperature Regime}

By Theorem \ref{thm:Finetti},
\[
m_{k_{1},\ldots,k_{M}}\approx\frac{\int_{\IR^{M}}\tanh^{{\bf k}}\left(y\right)e^{-NF(y)}\textup{d}y}{\int_{\IR^{M}}e^{-NF(y)}\textup{d}y},
\]
where we used the notation introduced in the proof of Proposition
\ref{prop:Laplace_multivariate}. To estimate these integrals for large $N$, we can use a standard
version of Laplace's method, since contrary to the situation in Proposition
\ref{prop:Laplace_multivariate}, the tanh function is not 0 at $\bar{m}$, the minimum of $F$
in the low temperature regime. Let $H(\bar{m})$ stand for the Hessian of $F$ at the minimum. We apply Theorem 3 on p.\! 495 of \cite{wong}
to the numerator and the denominator, which yields
\begin{align*}
m_{k_{1},\ldots,k_{M}} \approx \frac{\left( \frac{2\pi}{N} \right)^{M/2} m^{\bf k} \left( \det H(\bar{m}) \right)^{-1/2} \exp\left( -NF(\bar{m}) \right) }
 {\left( \frac{2\pi}{N} \right)^{M/2} \left( \det H(\bar{m}) \right)^{-1/2} \exp\left( -NF(\bar{m}) \right)} =m^{{\bf k}}.
\end{align*}
Note that in the numerator the function $g_0$ given in Theorem 3 in \cite{wong} is $g_0 (y) = \tanh^{\bf k} y$ for all $y \in \IR^M$, whereas in the denominator we chose the constant function $g_0 = 1$.

This concludes the calculation of the correlations. We will now use
these correlations to calculate the moments of the normalised sums
$\boldsymbol{S}$.
\end{rem}

\section{\label{sec:Moments}Moments of $\boldsymbol{S}$}

We divide this section into three subsections according to the regime
of the model.

\subsection{\label{subsec:High_Temp_Proof}High Temperature Regime}

We first prove the Law of Large Numbers \ref{thm:LLN_high_temp}.
In this case, the normalised sums are
\[
\boldsymbol{S}:=\left(\frac{S_{1}}{N_{1}},\ldots,\frac{S_{M}}{N_{M}}\right).
\]
Later on in this proof, we will distinguish between sum vectors of
different dimensions, in which case we will use a superindex in $\boldsymbol{S}^{M}$,
indicating that the vector is of dimension $M\in\mathbb{N}$.

We define $M_{K_{1},\ldots,K_{M}}$ to be the moments of order $\mathbf{K}=(K_{1},\ldots,K_{M})$
of this random vector, and set $K:=\sum_{i=1}^{M}K_{i}$. Our task
consists of calculating the large $N$ limit of these moments.

We can express the moments in the following fashion:
\begin{align}
M_{K_{1},\ldots,K_{M}} & =\frac{1}{N_{1}^{K_{1}}}\cdots\frac{1}{N_{M}^{K_{M}}}\sum\mathbb{E}\left(X_{1i_{11}}\cdots X_{Mi_{MK_{M}}}\right),\label{eq:moments_LLN_low_1}
\end{align}
where the sum in the last line is over all $i_{\nu1}\ldots,i_{\nu K_{\nu}}\in\{1,\ldots,N_{\nu}\}$
for each group $\nu$. This is where we use the correlations calculated
previously. The above sum has $N^{K}$ summands in it. We have to
find a way to reduce this to a more manageable number. We use Proposition
\ref{thm:comb-coeff-multiindex} to deal with the sum. The number
of multiindices $\underline{\mathbf{i}}=(i_{11},\ldots,i_{MK_{M}})$
with a given profile $\text{\ensuremath{\underbar{\ensuremath{\mathbf{r}}}}}=\text{\ensuremath{\underbar{\ensuremath{\rho}}}(\ensuremath{\underline{\mathbf{i}}})}$
is asymptotically
\begin{equation}
w_{\mathbf{K}}(\underbar{\ensuremath{\mathbf{r}}})\approx\prod_{\lambda=1}^{M}\frac{N_{\lambda}^{\sum_{l=1}^{L}r_{\lambda l}}}{r_{\lambda1}!r_{\lambda2}!\ldots r_{\lambda K_{\lambda}}!}\frac{K_{\lambda}!}{1!^{r_{\lambda1}}2!^{r_{\lambda2}}\cdots K_{\lambda}!^{r_{\lambda K_{\lambda}}}}.\label{eq:number_profiles_M}
\end{equation}
Thus we can state
\begin{equation}
\eqref{eq:moments_LLN_low_1}=\frac{1}{N_{1}}\cdots\frac{1}{N_{M}}\sum_{\text{\ensuremath{\underbar{\ensuremath{\mathbf{r}}}}}}w_{\mathbf{K}}(\underbar{\ensuremath{\mathbf{r}}})\mathbb{E}\left(X_{\text{\ensuremath{\underbar{\ensuremath{\mathbf{r}}}}}}\right),\label{eq:moments_LLN_low_2}
\end{equation}
where the sum is over all possible profile vectors $\underbar{\ensuremath{\mathbf{r}}}$.

The correlations are dependent on the coupling matrix. Since each
$X_{\lambda i}$ raised to an even power equals 1, we only focus on
the indices which occur an odd number of times.

By Theorems \ref{thm:EXY-hom} and \ref{thm:EXY-het_high}, the correlations
include powers of $N$ that depend on the number of indices occurring
an odd number of times. More precisely, up to a multiplicative constant, the correlations
are asymptotically
\[
N^{-\frac{1}{2}\left(\sum_{m_{1}=0}^{\left\lfloor K_{1}/2\right\rfloor }r_{1,2m_{1}+1}+\cdots+\sum_{m_{M}=0}^{\left\lfloor K_{M}/2\right\rfloor }r_{M,2m_{M}+1}\right)}.
\]
Therefore, in the sum above, each summand has the following power
of $N$:
\[
\frac{1}{N_{1}^{K_{1}}\cdots N_{M}^{K_{M}}}N_{1}^{\sum_{\ell_{1}=1}^{K_{1}}r_{1\ell_{1}}}\cdots N_{M}^{\sum_{\ell_{M}=1}^{K_{M}}r_{M\ell_{M}}}N^{-\frac{1}{2}\left(\sum_{m_{1}=0}^{\left\lfloor K_{1}/2\right\rfloor }r_{1,2m_{1}+1}+\cdots+\sum_{m_{M}=0}^{\left\lfloor K_{M}/2\right\rfloor }r_{M,2m_{M}+1}\right)}.
\]
Hence, the exponent of $N$ in each summand is given by
\begin{equation}
-K+\sum_{\lambda=1}^{M}\left(\sum_{m_{\lambda}=0}^{\left\lfloor K_{\lambda}/2\right\rfloor }r_{\lambda,2m_{\lambda}}+\frac{1}{2}\sum_{m_{\lambda}=0}^{\left\lfloor K_{\lambda}/2\right\rfloor }r_{\lambda,2m_{\lambda}+1}\right).\label{eq:power_N}
\end{equation}

\begin{lem}
For all $\text{\ensuremath{\underbar{\ensuremath{\mathbf{r}}}}}$,
the inequality $\eqref{eq:power_N}<0$ holds.
\end{lem}

\begin{proof}
We omit the straightforward proof.
\end{proof}
Since there are only finitely many (their number being independent of $N$) summands in (\ref{eq:moments_LLN_low_2}),
each of which converges to 0, the above lemma proves Theorem \ref{thm:LLN_high_temp}.

Now we turn to the CLT \ref{thm:CLT}. In the proof of the LLN above,
none of the summands contribute asymptotically to the moments, which
therefore converge to 0. If we normalise the sums $\boldsymbol{S}$
by a power of $N$ equal to $\frac{1}{2}$, we will see that the summands
with profile vectors belonging to $\Pi^{0(K_{\lambda})}$ (see Definition
\ref{def:profile_sets}) for each group $\lambda$ contribute positive
quantities to the moments. We proceed similarly to the above by defining
$M_{K_{1},\ldots,K_{M}}$ to be the moments of order $\mathbf{K}=(K_{1},\ldots,K_{M})$
of the random vector
\[
\boldsymbol{S}:=\left(\frac{S_{1}}{\sqrt{N_{1}}},\ldots,\frac{S_{M}}{\sqrt{N_{M}}}\right),
\]
and setting $K:=\sum_{i=1}^{M}K_{i}$. The moments can be expressed
as
\begin{align}
M_{K_{1},\ldots,K_{M}} & =\frac{1}{N_{1}^{\frac{K_{1}}{2}}}\cdots\frac{1}{N_{M}^{\frac{K_{M}}{2}}}\sum_{\text{\ensuremath{\underbar{\ensuremath{\mathbf{r}}}}}}w_{\mathbf{K}}(\underbar{\ensuremath{\mathbf{r}}})\mathbb{E}\left(X_{\text{\ensuremath{\underbar{\ensuremath{\mathbf{r}}}}}}\right).\label{eq:moments_CLT_1}
\end{align}
Asymptotically, the powers of $N$ in (\ref{eq:moments_CLT_1}) are
\[
\frac{1}{N_{1}^{\frac{K_{1}}{2}}}\cdots\frac{1}{N_{M}^{\frac{K_{M}}{2}}}N_{1}^{\sum_{l_{1}=1}^{K_{1}}r_{1l_{1}}}\cdots N_{M}^{\sum_{l_{M}=1}^{K_{M}}r_{Ml_{M}}}N^{-\frac{1}{2}\left(\sum_{m_{1}=0}^{\left\lfloor K_{1}/2\right\rfloor }r_{1,2m_{1}+1}+\cdots+\sum_{m_{M}=0}^{\left\lfloor K_{M}/2\right\rfloor }r_{M,2m_{M}+1}\right)}.
\]
We separate the sum in (\ref{eq:moments_CLT_1}) into two summands:
$A_{1}$, where the sum runs over $\underbar{\ensuremath{\mathbf{r}}}$
such that for all $\lambda$ $\ensuremath{\underline{r_{\lambda}}\in\Pi^{0(K_{\lambda})}}$,
and $A_{2}$ for all other $\underbar{\ensuremath{\mathbf{r}}}$.
We claim
\begin{prop}
\label{prop:Pi_plus_no_contrib}The limit of $A_{2}$ as $N\rightarrow\infty$
is $0$.
\end{prop}

\begin{proof}
Let $\underbar{\ensuremath{\mathbf{r}}}$ be such that $\underline{r_{\nu}}\in\Pi^{+(K_{\nu})}$
for some group $\nu$.

We have for any group $\lambda$ the factor
\[
\frac{1}{N_{\lambda}^{\frac{K_{\lambda}}{2}}}N_{\lambda}^{\sum_{l_{\lambda}=1}^{K_{\lambda}}r_{\lambda l_{\lambda}}}N^{-\frac{1}{2}\sum_{m_{\lambda}=0}^{\left\lfloor K_{\lambda}/2\right\rfloor }r_{\lambda,2m_{\lambda}+1}},
\]
which has a power of $N$ equal to
\begin{align*}
-\frac{K_{\lambda}}{2}+\sum_{l_{\lambda}=1}^{K_{\lambda}}r_{\lambda l_{\lambda}}-\frac{1}{2}\sum_{m_{\lambda}=0}^{\left\lfloor K_{\lambda}/2\right\rfloor }r_{\lambda,2m_{\lambda}+1} & =\frac{1}{2}\left(-K_{\lambda}+\sum_{m_{\lambda}=0}^{\left\lfloor K_{\lambda}/2\right\rfloor }2r_{\lambda,2m_{\lambda}}+\sum_{m_{\lambda}=0}^{\left\lfloor K_{\lambda}/2\right\rfloor }r_{\lambda,2m_{\lambda}+1}\right)\leq0,
\end{align*}
with equality if and only if
\[
\sum_{m_{\lambda}=0}^{\left\lfloor K_{\lambda}/2\right\rfloor }2r_{\lambda,2m_{\lambda}}+\sum_{m_{\lambda}=0}^{\left\lfloor K_{\lambda}/2\right\rfloor }r_{\lambda,2m_{\lambda}+1}=\sum_{l_{\lambda}=1}^{K_{\lambda}}\ell_{\lambda}r_{\lambda l_{\lambda}}=K_{\text{\ensuremath{\lambda}}}.
\]
The first equality above holds if and only if $r_{\lambda,\ell_{\lambda}}=0$
for all $\ell_{\lambda}>2$.

For the group $\nu$, we have by assumption some $k>2$ with $r_{\nu,k}>0$,
and hence $r_{\nu,k}<kr_{\nu,k}$. That implies
\[
-\frac{K_{\nu}}{2}+\sum_{\ell_{\nu}=1}^{K_{\nu}}r_{1\ell_{\nu}}-\frac{1}{2}\sum_{m_{\nu}=0}^{\left\lfloor K_{\nu}/2\right\rfloor }r_{\nu,2m_{\nu}+1}<K_{\nu}.
\]
So the overall power of $N$ in each summand is negative and it converges
to 0 as $N\rightarrow\infty$. There are only finitely many such summands
in $A_{2}$, and we conclude that $A_{2}$ converges to 0 as $N\rightarrow\infty$.
\end{proof}
It follows from this proposition that
\begin{align}
M_{K_{1},\ldots,K_{M}}\approx A_{1} & =\frac{1}{N_{1}^{\frac{K_{1}}{2}}}\cdots\frac{1}{N_{M}^{\frac{K_{M}}{2}}}\sum_{\text{\ensuremath{\underbar{\ensuremath{\mathbf{r}}}}:\ensuremath{\forall\lambda:\underline{r_{\lambda}}\in\Pi^{0(K_{\lambda})}}}}w_{\mathbf{K}}(\underbar{\ensuremath{\mathbf{r}}})\mathbb{E}\left(X_{\text{\ensuremath{\underbar{\ensuremath{\mathbf{r}}}}}}\right)\nonumber \\
 & \approx\frac{1}{N_{1}^{\frac{K_{1}}{2}}}\cdots\frac{1}{N_{M}^{\frac{K_{M}}{2}}}\sum_{\text{\ensuremath{\underbar{\ensuremath{\mathbf{r}}}}:\ensuremath{\forall\lambda:\underline{r_{\lambda}}\in\Pi^{0(K_{\lambda})}}}}\prod_{\lambda=1}^{M}\frac{N_{\lambda}^{\sum_{\ell=1}^{L}r_{\lambda\ell}}}{r_{\lambda1}!r_{\lambda2}!}\frac{K_{\lambda}!}{2!^{r_{\lambda2}}}\mathbb{E}\left(X_{\text{\ensuremath{\underbar{\ensuremath{\mathbf{r}}}}}}\right)\nonumber \\
 & =\frac{1}{N_{1}^{\frac{K_{1}}{2}}}\cdots\frac{1}{N_{M}^{\frac{K_{M}}{2}}}\sum_{k_{1}=0}^{K_{1}}\cdots\sum_{k_{M}=0}^{K_{M}}\prod_{\lambda=1}^{M}\frac{N_{\lambda}^{\frac{K_{\lambda}+k_{\lambda}}{2}}}{k_{\lambda}!\left(\frac{K_{\lambda}-k_{\lambda}}{2}\right)!}\frac{K_{\lambda}!}{2^{\frac{K_{\lambda}-k_{\lambda}}{2}}}\mathbb{E}\left(X_{\text{\ensuremath{\underbar{\ensuremath{\mathbf{r}}}}}}\right).\label{eq:S_moments_A_1}
\end{align}
In the last line, each sum is over those $k_{\lambda}$ which have
the same parity as $K_{\lambda}$ and $\underbar{\ensuremath{\mathbf{r}}}$
is the profile vector where $\underline{r_{\lambda}}=\left(k_{\lambda},\left(K_{\lambda}-k_{\lambda}\right)/2\right)$
for all groups $\lambda$.

Of course, the expectation $\mathbb{E}\left(X_{\text{\ensuremath{\underbar{\ensuremath{\mathbf{r}}}}}}\right)$
depends on the class of coupling matrix. As a reminder, if $J$ is
homogeneous, then
\begin{equation}
\mathbb{E}\left(X_{\text{\ensuremath{\underbar{\ensuremath{\mathbf{r}}}}}}\right)\approx(k_{1}+\cdots+k_{M}-1)!!\left(\frac{\beta}{1-\beta}\right)^{\frac{k_{1}+\cdots+k_{M}}{2}}\frac{1}{N^{\frac{k_{1}+\cdots+k_{M}}{2}}}.\label{eq:EXY_hom_high}
\end{equation}
If, on the other hand, $J$ is heterogeneous, then
\begin{equation}
\mathbb{E}\left(X_{\text{\ensuremath{\underbar{\ensuremath{\mathbf{r}}}}}}\right)\approx m_{k_{1},\ldots,k_{M}}\left(H^{-1}\right)\frac{1}{N^{\frac{k_{1}+\cdots+k_{M}}{2}}}.\label{eq:EXY_het_high}
\end{equation}
Recall that in the last line above $m_{k_{1},\ldots,k_{M}}\left(H^{-1}\right)$
stands for the moment of centred multivariate normal distributions
with covariance matrix $H^{-1}$.

We now show that the correlation in (\ref{eq:EXY_hom_high}) can be
expressed in a similar fashion as in (\ref{eq:EXY_het_high}).
\begin{defn}
Let for all $n\in\mathbb{N}$ and all $\ell_{1},\ldots,\ell_{n}\in\mathbb{N}_{0}$
the notation $m_{\ell_{1},\ldots,\ell_{n}}^{n}(C)$ stand for an $n$-dimensional
multivariate moment of an $n$-dimensional multivariate normal distribution
with mean 0 and covariance matrix $C$. The superindex $n$ indicates
the dimension of the distribution.

We state Isserlis's Theorem which we will use in the proof of the
next proposition.
\end{defn}

\begin{thm}
[Isserlis's Theorem]\label{thm:Isserlis}Let $\ell_{1},\ldots,\ell_{n}\in\mathbb{N}_{0}$.
Assume $(Z_{1}\ldots,Z_{n})$ follows $\mathcal{N}(0,C)$, $C=(c_{ij})_{i,j=1,\ldots,n}$.
Define a set $\boldsymbol{Z}=\bigcup_{i=1}^{M+1}\left\{ Z_{i1},\ldots,Z_{i\ell_{i}}\right\} $
with $\ell_{i}$ copies of $Z_{i}$ for each $i=1,\ldots,n$. If the
sum $\ell_{1}+\cdots+\ell_{n}$ is odd, the moment $m_{\ell_{1},\ldots,\ell_{n}}^{n}(C)$
is $0$ due to the symmetry of the normal distribution. Otherwise
we let $\mathcal{P}$ be the set of all possible pair partitions of
the set $\boldsymbol{Z}$. Let for any $\pi\in\mathcal{P}$ $\pi(k)=(i,j)$
if the $k$-th pair in $\pi$ has one copy of $Z_{i}$ and one copy of $Z_{j}$. Then
\begin{equation*}
m_{\ell_{1},\ldots,\ell_{n}}^{n}(C)=\sum_{\pi\in\mathcal{P}}\prod_{k=1}^{\frac{\ell_{1}+\cdots+\ell_{n}}{2}}c_{\pi(k)}.
\end{equation*}
\end{thm}

\begin{proof}
For a proof see the original publication \cite{Isserlis}.
\end{proof}
\begin{prop}
Let $\ell_{1},\ldots,\ell_{n}\in\mathbb{N}_{0}$. Let $C=(c)_{\lambda,\mu=1,\ldots,n}$
be a homogeneous matrix (recall Definition \ref{def:hom_matrix}).
Then
\[
m_{\ell_{1},\ldots,\ell_{n}}^{n}(C)=m_{\ell_{1}+\cdots+\ell_{n}}^{1}(c).
\]
\end{prop}

\begin{proof}
By Isserlis's Theorem,
\begin{align*}
m_{\ell_{1},\ldots,\ell_{n}}^{n}(C) & =\sum_{\pi\in\mathcal{P}}\prod_{k=1}^{\frac{\ell_{1}+\cdots+\ell_{n}}{2}}c_{\pi(k)}=\sum_{\pi\in\mathcal{P}}c^{\frac{\ell_{1}+\cdots+\ell_{n}}{2}}\\
 & =(\ell_{1}+\cdots+\ell_{n}-1)!!\,c^{\frac{\ell_{1}+\cdots+\ell_{n}}{2}}=m_{\ell_{1}+\cdots+\ell_{n}}^{1}(c).
\end{align*}
\end{proof}
\begin{cor}
The homogeneous coupling correlation (\ref{eq:EXY_hom_high}) is equal
to
\begin{align*}
\mathbb{E}\left(X_{\text{\ensuremath{\underbar{\ensuremath{\mathbf{r}}}}}}\right) & \approx m_{k_{1},\ldots,k_{M}}^{M}\left(\left(\bar{\beta}\right)_{\lambda,\mu=1,\ldots,M}\right)\frac{1}{N^{\frac{k_{1}+\cdots+k_{M}}{2}}},
\end{align*}
where $\bar{\beta}=\frac{\beta}{1-\beta}$.
\end{cor}

We can therefore treat the two classes of coupling matrix $J$ simultaneously
by setting
\begin{equation}
\Sigma:=(\sigma_{\lambda,\mu})_{\lambda,\mu=1,\ldots,M}:=\begin{cases}
\bar{\beta}, & \text{if }J\text{ is homogeneous,}\\
\left(H^{-1}\right)_{\lambda,\mu}, & \text{if }J\text{ is heterogeneous.}
\end{cases}\label{eq:cov_matrix_Sigma}
\end{equation}
Next, we present a recursive formulation of Isserlis's Theorem.
\begin{prop}
Let $M\in\mathbb{N},l_{1},\ldots,l_{M+1}\in\mathbb{N}_{0}$, and $m_{l_{1},\ldots,l_{M+1}}^{M+1}$
be the moment of order $(l_{1},\ldots,l_{M+1})$ of an $(M+1)$-dimensional
multivariate normal distribution with covariance matrix
\[
\Sigma^{M+1}=(\sigma_{i,j})_{i,j=1,\ldots,M+1}.
\]
Then the moment $m_{l_{1},\ldots,l_{M+1}}^{M+1}$ can be calculated
in terms of the covariances and the $M$-dimensional moments $m_{k_{1},\ldots,k_{M}}^{M}$
of the $M$-dimensional multivariate normal distribution of the first $M$ entries with covariance
matrix
\[
\Sigma^{M}=(\sigma_{i,j})_{i,j=1,\ldots,M}
\]
as follows:
\begin{align}
m_{l_{1},\ldots,l_{M+1}}^{M+1} & =\sum_{(r_{1},\ldots,r_{M})\in A}\frac{l_{1}!\cdots l_{M+1}!}{\left(\frac{l_{M+1}-\sum_{i=1}^{M}r_{i}}{2}\right)!2^{\frac{l_{M+1}-\sum_{i=1}^{M}r_{i}}{2}}r_{1}!\cdots r_{M}!(l_{1}-r_{1})!\cdots(l_{M}-r_{M})!} \nonumber \\
 & \quad\cdot\sigma_{1,M+1}^{r_{1}}\cdots\sigma_{M,M+1}^{r_{M}}\sigma_{M+1,M+1}^{\frac{l_{M+1}-\sum_{i=1}^{M}r_{i}}{2}}m_{l_{1}-r_{1},\ldots,l_{M}-r_{M}}^{M},\label{eq:rec_Isserlis}
\end{align}
where $A=\left\{ (r_{1},\ldots,r_{M})\;|\;\forall i=1,\ldots,M:0\leq r_{i}\leq l_{i}\wedge l_{M+1}\mathcal{\textrm{ and }}l_{i}-r_{i},l_{M+1}-\sum_{i=1}^{M}r_{i}\text{ are even}\right\} $.
\end{prop}

\begin{proof}
Let $(Z_{1}\ldots,Z_{M+1})$ follow $\mathcal{N}(0,\Sigma^{M+1})$.
By Isserlis's Theorem, the higher order moments of an $(M+1)$-dimensional
multivariate normal distribution can be calculated as the sum over
all possible pair partitions of the component variables. To calculate
the moment
\[
m_{l_{1},\ldots,l_{M+1}}^{M+1}:=m_{l_{1},\ldots,l_{M+1}}^{M+1}\left(\mathcal{N}\left(0,\Sigma^{M+1}\right)\right)
\]
of order $(l_{1},\ldots,l_{M+1})$, we define a set $\boldsymbol{Z}=\bigcup_{i=1}^{M+1}\left\{ Z_{i1},\ldots,Z_{i\ell_{i}}\right\} $
with $l_{i}$ copies of $Z_{i}$ for each $i=1,\ldots,M+1$. If the
sum $l_{1}+\cdots+l_{M+1}$ is odd, the moment $m_{l_{1},\ldots,l_{M+1}}^{M+1}$
is 0 due to the symmetry of the normal distribution. Otherwise we
let $\mathcal{P}$ be the set of all possible pair partitions of the
set $\boldsymbol{Z}$. Let for any $\pi\in\mathcal{P}$ $\pi(k)=(i,j)$
if the $k$-th pair in $\pi$ has one copy $Z_{i}$ and one copy $Z_{j}$. Then, by Isserlis,
\begin{equation}
m_{l_{1},\ldots,l_{M+1}}^{M+1}=\sum_{\pi\in\mathcal{P}}\prod_{k=1}^{\frac{l_{1}+\cdots+l_{M+1}}{2}}\sigma_{\pi(k)}\label{eq:Isserlis-M_plus_1}
\end{equation}
holds. We express the $(M+1)$-dimensional moment $m_{l_{1},\ldots,l_{M+1}}^{M+1}$
in terms of the $M$-dimensional moments $m_{l_{1},\ldots,l_{M+1}}^{M}$.
Each of the $l_{M+1}$ copies of $Z_{M+1}$ can be paired up with
a copy of $Z_{1},\ldots,Z_{M}$ or another copy of $Z_{M+1}$. Let
the number of copies of $Z_{M+1}$ paired up with a $Z_{i}$ for $i=1,\ldots,M$
be $r_{i}$. Then there are $r:=\sum_{i=1}^{M}r_{i}$ remaining copies
of $Z_{M+1}$ that must be paired, hence $l_{M+1}-r$ must be even.
Applying Isserlis to the $M$-dimensional moment, a single pair partition
$\pi$ contributes
\[
\sigma_{1,M+1}^{r_{1}}\cdots\sigma_{M,M+1}^{r_{M}}\sigma_{M+1,M+1}^{\frac{l_{M+1}-\sum_{i=1}^{M}r_{i}}{2}}m_{l_{1}-r_{1},\ldots,l_{M}-r_{M}}^{M}
\]
to (\ref{eq:Isserlis-M_plus_1}). How many of these pair partitions
$\pi$ for a given profile $(r_{1},\ldots,r_{M})$ are there?

For each $i=1,\ldots,M$, we have to pick $r_{i}$ from the $l_{M+1}$
copies of $Z_{M+1}$ to be paired with a $Z_{i}$ and $l_{M+1}-r$
to be paired with another $Z_{M+1}:\left(\begin{array}{c}
l_{M+1}\\
r_{1},\ldots,r_{M},l_{M+1}-r
\end{array}\right).$

We have to select for each $i=1,\ldots,M$ the corresponding $r_{i}$
copies of $Z_{i}:\left(\begin{array}{c}
l_{i}\\
r_{i}
\end{array}\right).$

For each $i=1,\ldots,M$, the selected $r_{i}$ $Z_{i}$ and $r_{i}$
$Z_{M+1}$ can be paired in $r_{i}!$ ways.

The remaining $(l_{M+1}-r)$ $Z_{M+1}$ can be ordered into $(l_{M+1}-r-1)!!=\frac{(l_{M+1}-r)!}{\left(\frac{l_{M+1}-r}{2}\right)!\,2^{\frac{l_{M+1}-r}{2}}}$
pairs. Multiplying all of these terms together and simplifying, we
obtain
\[
\frac{l_{1}!\cdots l_{M+1}!}{\left(\frac{l_{M+1}-\sum_{i=1}^{M}r_{i}}{2}\right)!\,2^{\frac{l_{M+1}-\sum_{i=1}^{M}r_{i}}{2}}r_{1}!\cdots r_{M}!(l_{1}-r_{1})!\cdots(l_{M}-r_{M})!}.
\]
Then we sum over all profiles $(r_{1},\ldots,r_{M})$ to obtain the
desired result.
\end{proof}
Given the covariance matrix $\Sigma$ defined in (\ref{eq:cov_matrix_Sigma}),
we set
\[
C:=I+\sqrt{\boldsymbol{\alpha}}\Sigma\sqrt{\boldsymbol{\alpha}}.
\]
Before proceeding with the proof of the CLT, we state an auxiliary
lemma.
\begin{lem}
\label{lem:summation_switch}Let $(L_{1},\ldots,L_{M+1})$ be a fixed
$(M+1)$-tuple of non-negative integers. Let $\{(r,r_{1},\ldots,r_{M},l_{1},\ldots l_{M+1})\}$
be a set of tuples of non-negative integers such that $r=\sum_{i=1}^{M}r_{i}$.
We define the following conditions:
\begin{align}
0 & \leq l_{i}\leq L_{i}\quad(i=1,\ldots,M+1)\label{eq:i1}\\
0 & \leq r\leq l_{M+1}\label{eq:i2}\\
0 & \leq r_{i}\leq l_{i}\wedge l_{M+1}\quad(i=1,\ldots,M)\label{eq:i3}\\
0 & \leq r_{i}\leq L_{i}\quad(i=1,\ldots,M)\label{eq:i4}\\
0 & \leq r\leq L_{M+1}\label{eq:i5}\\
r_{i} & \leq l_{i}\leq L_{i}\quad(i=1,\ldots,M)\label{eq:i6}\\
r & \leq l_{M+1}\leq L_{M+1}.\label{eq:i7}
\end{align}
Let
\begin{align*}
A & :=\{(r,r_{1},\ldots,r_{M},l_{1},\ldots l_{M+1})\,|\,\eqref{eq:i1},\eqref{eq:i2},\eqref{eq:i3}\text{ hold}\},\\
B & :=\{(r,r_{1},\ldots,r_{M},l_{1},\ldots l_{M+1})\,|\,\eqref{eq:i4},\eqref{eq:i5},\eqref{eq:i6},\eqref{eq:i7}\text{ hold}\}.
\end{align*}
Then we have $A=B$.
\end{lem}

\begin{proof}
We omit the proof which is a straightforward verification.
\end{proof}
Now we are prepared to prove the CLT.
\begin{proof}
[Proof of Theorem \ref{thm:CLT}]The proof proceeds by induction on
$M$. For $M=1$, we show that $\frac{S_{1}}{\sqrt{N_{1}}}$ is asymptotically
univariate normally distributed with variance $1+\alpha_{1}\sigma_{11}=C$, where $\sigma_{11}$ is the sole entry of the covariance matrix $\Sigma$ defined in equation (\ref{eq:cov_matrix_Sigma}), by the method of moments.

Let $K\in\mathbb{N}_{0}$. Then the moment $M_{K}^{1}$ is 0 if $K$
is odd. Otherwise it is asymptotically equal to
\begin{align*}
\sum_{k=0}^{K/2}\frac{\alpha_{1}^{k}}{(2k)!\left(\frac{K}{2}-k\right)!}\frac{K!}{2^{\frac{K}{2}-k}}m_{2k} & =\sum_{k=0}^{K/2}\frac{\alpha_{1}^{k}}{(2k)!\left(\frac{K}{2}-k\right)!}\frac{K!}{2^{\frac{K}{2}-k}}\frac{(2k)!}{k!2^{k}}\sigma_{11}^{k}\\
 & =\frac{K!}{(\frac{K}{2})!2^{\frac{K}{2}}}\sum_{k=0}^{K/2}\frac{\left(\frac{K}{2}\right)!(\alpha_{1}\sigma_{11})^{k}}{k!\left(\frac{K}{2}-k\right)!}=(K-1)!!(1+\alpha_{1}\sigma_{11})^{\frac{K}{2}}.
\end{align*}
Therefore, the claim holds for $M=1$. Let for the rest of this proof
\[
m_{l_{1},\ldots,l_{M+1}}^{M+1}:=m_{l_{1},\ldots,l_{M+1}}^{M+1}(\Sigma).
\]

Assume the claim holds for some $M\in\mathbb{N}$ and we show that
it also holds for $M+1$ by proving that for all $L_{1},\ldots,L_{M+1}\in\mathbb{N}_{0}$
the moments of order $(L_{1},\ldots,L_{M+1})$ satisfy the recursive
relation
\begin{align}
M_{L_{1},\ldots,L_{M+1}}^{M+1} & =\sum_{(r_{1},\ldots,r_{M})\in A}\frac{L_{1}!\cdots L_{M+1}!}{\left(\frac{L_{M+1}-\sum_{i=1}^{M}r_{i}}{2}\right)!2^{\frac{L_{M+1}-\sum_{i=1}^{M}r_{i}}{2}}r_{1}!\cdots r_{M}!(L_{1}-r_{1})!\cdots(L_{M}-r_{M})!} \nonumber \\
 & \quad\cdot(\sqrt{\alpha_{1}\alpha_{M+1}}\sigma_{1,M+1})^{r_{1}}\cdots(\sqrt{\alpha_{M}\alpha_{M+1}}\sigma_{M,M+1})^{r_{M}} \nonumber \\
 & \quad\cdot(1+\alpha_{M}\sigma_{M+1,M+1})^{\frac{L_{M+1}-\sum_{i=1}^{M}r_{i}}{2}}M_{L_{1}-r_{1},\ldots,L_{M}-r_{M}}^{M},\label{eq:M_plus_1_rec}
\end{align}
where $A$ is given in Lemma \ref{lem:summation_switch}. If the sum
$L_{1}+\cdots+L_{M+1}$ is odd, the moment $M_{L_{1},\ldots,L_{M+1}}^{M+1}$
is 0. If the sum is even, $M_{L_{1},\ldots,L_{M+1}}^{M+1}$ is asymptotically equal to
\[
\sum_{l_{1}=0}^{L_{1}}\cdots\sum_{l_{M+1}=0}^{L_{M+1}}\frac{L_{1}!\alpha_{1}^{\frac{l_{1}}{2}}}{l_{1}!\left(\frac{L_{1}-l_{1}}{2}\right)!2^{\frac{L_{1}-l_{1}}{2}}}\cdots\frac{L_{M+1}!\alpha_{M+1}^{\frac{l_{M+1}}{2}}}{l_{M+1}!\left(\frac{L_{M+1}-l_{M+1}}{2}\right)!2^{\frac{L_{M+1}-l_{M+1}}{2}}}m_{l_{1},\ldots,l_{M+1}}^{M+1},
\]
where each sum goes over $l_{i}$ with values with the same parity
as $L_{i}$. We next express $m_{l_{1},\ldots,l_{M+1}}^{M+1}$, which
is a moment of an $(M+1)$-dimensional normal distribution recursively
according to (\ref{eq:rec_Isserlis}). Then we switch summation indices
from $A$ to $B$ as defined by Lemma \ref{lem:summation_switch}:
\begin{align*}
= & \sum_{r=0}^{L_{M+1}}\sum_{(r_{1},\ldots,r_{M})}\sum_{l_{1}=r_{1}}^{L_{1}}\cdots\sum_{l_{M}=r_{M}}^{L_{M}}\sum_{l_{M+1}=r}^{L_{M+1}}\frac{L_{1}!\alpha_{1}^{\frac{l_{1}}{2}}}{\left(\frac{L_{1}-l_{1}}{2}\right)!2^{\frac{L_{1}-l_{1}}{2}}}\cdots\frac{L_{M+1}!\alpha_{M+1}^{\frac{l_{M+1}}{2}}}{\left(\frac{L_{M+1}-l_{M+1}}{2}\right)!2^{\frac{L_{M+1}-l_{M+1}}{2}}}\\
 & \cdot\frac{\sigma_{1,M+1}^{r_{1}}\cdots\sigma_{M,M+1}^{r_{M}}\sigma_{M+1,M+1}^{\frac{l_{M+1}-r}{2}}m_{l_{1}-r_{1},\ldots,l_{M}-r_{M}}^{M}}{\left(\frac{l_{M+1}-r}{2}\right)!2^{\frac{l_{M+1}-r}{2}}r_{1}!\cdots r_{M}!(l_{1}-r_{1})!\cdots(l_{M}-r_{M})!}.
\end{align*}
Now we switch indices once more: define for all $i=1,\ldots,M$ $s_{i}:=l_{i}-r_{i}$
and $s:=l_{M+1}-r$. After some lengthy calculations, we see that
the sum is
\begin{align*}
= & \sum_{r=0}^{L_{M+1}}\sum_{(r_{1},\ldots,r_{M})}\frac{L_{1}!\cdots L_{M}!}{r_{1}!\cdots r_{M}!(L_{1}-r_{1})!\cdots(L_{M}-r_{M})!\left(\frac{L_{M+1}-r}{2}\right)!2^{\frac{L_{M+1}-r}{2}}}\\
 & \cdot(\sqrt{\alpha_{1}\alpha_{M+1}}\sigma_{1,M+1})^{r_{1}}\cdots(\sqrt{\alpha_{M}\alpha_{M+1}}\sigma_{M,M+1})^{r_{M}}(1+\alpha_{M+1}\sigma_{M+1,M+1})^{\frac{L_{M+1}-r}{2}}M_{L_{1}-r_{1},\ldots,L_{M}-r_{M}}^{M}.
\end{align*}
As we can see, this last expression is equal to (\ref{eq:M_plus_1_rec}),
and therefore the recursive relation holds for the moments $M^{M+1}$
in terms of the moments $M^{M}$ holds. This implies that if for some
$M\in\mathbb{N}$ $\boldsymbol{S}^{M}$ is asymptotically normally
distributed, then so is $\boldsymbol{S}^{M+1}$. This concludes the
proof by induction.
\end{proof}
We have thus shown all results for the high temperature regime.

\subsection{Critical Regime}

In the critical regime, we proceed similarly to the above analysis
of the high temperature regime. We first prove the Law of Large Numbers
\ref{thm:LLN_crit_reg} and then the fluctuations results Theorems
\ref{Fluctuations_hom} and \ref{Fluctuations_het}.

We use the previously calculated correlations presented in Theorems
\ref{thm:EXY-hom} and \ref{thm:EXY-het_crit}. For the Law of Large
Numbers, the moments can be expressed analogously to Section \ref{subsec:High_Temp_Proof},
and the proof that these moments converge to 0 as $N\rightarrow\infty$ is
identical to the proof found there.

For the fluctuations results, we normalise the sums $\boldsymbol{S}$
by a power of $N$ equal to $\frac{3}{4}$. Then, we will see that
the summands with profile vectors belonging to $\Pi^{0(K_{\lambda})}$
for each group $\lambda$ contribute positive quantities to the moments.
We proceed similarly to the above by defining $M_{K_{1},\ldots,K_{M}}$
to be the moments of order $\mathbf{K}=(K_{1},\ldots,K_{M})$ of the
random vector
\[
\boldsymbol{S}:=\left(\frac{S_{1}}{N_{1}^{\frac{3}{4}}},\ldots,\frac{S_{M}}{N_{M}^{\frac{3}{4}}}\right),
\]
and setting $K:=\sum_{i=1}^{M}K_{i}$. The moments are
\begin{align}
M_{K_{1},\ldots,K_{M}} & =\frac{1}{N_{1}^{\frac{3K_{1}}{4}}}\cdots\frac{1}{N_{M}^{\frac{3K_{M}}{4}}}\sum_{\text{\ensuremath{\underbar{\ensuremath{\mathbf{r}}}}}}w_{\mathbf{K}}(\underbar{\ensuremath{\mathbf{r}}})\mathbb{E}\left(X_{\text{\ensuremath{\underbar{\ensuremath{\mathbf{r}}}}}}\right).\label{eq:moments_crit_fluct}
\end{align}
We separate the sum in (\ref{eq:moments_crit_fluct}) into two summands:
$A_{1}$, where $\underbar{\ensuremath{\mathbf{r}}}$ satisfies for
all $\lambda$ $\underline{r_{\lambda}}\in\Pi_{K_{\lambda}}^{(K_{\lambda})}$,
and $A_{2}$ all other $\underbar{\ensuremath{\mathbf{r}}}$. This
means $A_{1}$ contains all those profile vectors that describe multiindices
with no repeating indices at all, and $A_{2}$ all other profile vectors.

We claim
\begin{prop}
The limit of $A_{2}$ as $N\rightarrow\infty$ is $0$.
\end{prop}

\begin{proof}
We omit the proof, which proceeds along the same lines as before.
\end{proof}
It follows from this proposition that
\begin{align}
M_{K_{1},\ldots,K_{M}} & \approx A_{1}=\frac{1}{N_{1}^{\frac{3K_{1}}{4}}}\cdots\frac{1}{N_{M}^{\frac{3K_{M}}{4}}}w_{\mathbf{K}}(\underbar{\ensuremath{\mathbf{r}}})\mathbb{E}\left(X_{\text{\ensuremath{\underbar{\ensuremath{\mathbf{r}}}}}}\right),\label{eq:moments_crit_fluct_2}
\end{align}
where the profile vector $\underbar{\ensuremath{\mathbf{r}}}$ in
the last expression is the one that has $r_{\lambda1}=K_{\lambda}$
for all groups $\lambda$. The correlations $\mathbb{E}\left(X_{\text{\ensuremath{\underbar{\ensuremath{\mathbf{r}}}}}}\right)$
have the same powers of $N$ but different constants that depend on
$K_{\lambda}$ in the homogeneous and the heterogeneous case. We set
the constants $c_{K_{1},\ldots,K_{M}}$ equal to
\begin{equation}
\begin{cases}
12^{\frac{K_{1}+\cdots+K_{M}}{2}}\frac{\Gamma\left(\frac{K_{1}+\cdots+K_{M}+1}{4}\right)}{\Gamma\left(\frac{1}{4}\right)}, & \text{if }J\text{ is homogeneous,}\\
\left[\frac{12}{\alpha_{1}(L_{22}-\alpha_{2})^{2}+\alpha_{2}(L_{11}-\alpha_{1})^{2}}\right]^{\frac{K+Q}{4}}\\
\cdot(L_{11}-\alpha_{1})^{\frac{Q}{2}}(L_{22}-\alpha_{2})^{\frac{K}{2}}\frac{\Gamma\left(\frac{K+Q+1}{4}\right)}{\Gamma\left(\frac{1}{4}\right)}, & \text{if }M=2\text{ and }J\text{ is heterogeneous.}
\end{cases}\label{eq:constants_crit}
\end{equation}
With this definition,
\[
\eqref{eq:moments_crit_fluct_2}\approx c_{K_{1},\ldots,K_{M}}\alpha_{1}^{\frac{K_{1}}{4}}\cdots\alpha_{M}^{\frac{K_{M}}{4}},
\]
which proves Theorems \ref{Fluctuations_hom} and \ref{Fluctuations_het}.

\subsubsection{Proof of Theorem \ref{thm:Crit_Lin_Transform}}

We prove Theorem \ref{thm:Crit_Lin_Transform}. In a departure from
the method of moments, we use the de Finetti representation to show
the convergence of the sequence of characteristic functions of the
transformations given in Theorem \ref{thm:Crit_Lin_Transform}. We
prove the statement
\[
T_{N}:=\frac{\sqrt{L_{1}-\alpha_{1}}}{\sqrt{\alpha_{1}N_{1}}}S_{1}-\frac{\sqrt{L_{2}-\alpha_{2}}}{\sqrt{\alpha_{2}N_{2}}}S_{2}\underset{N\to\infty}{\Longrightarrow}\mathcal{N}\left(0,1+\frac{L_{1}-\alpha_{1}}{\alpha_{1}}+\frac{L_{2}-\alpha_{2}}{\alpha_{2}}\right)
\]
for heterogeneous coupling. The other statements can be shown analogously.

We set $\gamma_{\lambda}:=\sqrt{L_{\lambda}-\alpha_{\lambda}}$ for
both groups. The characteristic function of $T_{N}$ is defined as
\[
\varphi_{N}\left(t\right):=\mathbb{E}\left(\exp\left(it\left(\frac{\gamma_{1}}{\sqrt{\alpha_{1}N_{1}}}S_{1}-\frac{\gamma_{2}}{\sqrt{\alpha_{2}N_{2}}}S_{2}\right)\right)\right).
\]
We have to prove that $\varphi_{N}$ converges pointwise to the characteristic
function $\varphi\left(t\right)=\exp\left(-\left(1+\frac{\gamma_{1}^{2}}{\alpha_{1}}+\frac{\gamma_{2}^{2}}{\alpha_{2}}\right)\frac{t^{2}}{2}\right)$
of a centred normal distribution with variance equal to $1+\frac{\gamma_{1}^{2}}{\alpha_{1}}+\frac{\gamma_{2}^{2}}{\alpha_{2}}$.

By Theorem \ref{thm:Finetti}, the characteristic function
$\varphi_{N}$ can be expressed as
\begin{align*}
\varphi_{N}\left(t\right) & =Z^{-1}\int_{\mathbb{R}^{M}}E_{y}\left(\exp\left(it\left(\frac{\gamma_{1}}{\sqrt{\alpha_{1}N_{1}}}S_{1}-\frac{\gamma_{2}}{\sqrt{\alpha_{2}N_{2}}}S_{2}\right)\right)\right)\\
 & \quad\cdot e^{-N\left(\frac{1}{2}y^{T}Ly-\sum_{\lambda}\frac{N_{\lambda}}{N}\ln\cosh y_{\lambda}\right)}\mathrm{d}y.
\end{align*}
The normalisation constant $Z^{-1}$ is equal to $\int_{\mathbb{R}^{M}}e^{-N\left(\frac{1}{2}y^{T}Ly-\sum_{\lambda}\frac{N_{\lambda}}{N}\ln\cosh y_{\lambda}\right)}\mathrm{d}y$.
The transformations we will be doing to the integral above are equally
applied to $Z^{-1}$. Hence, we will be omitting multiplicative constants
that stem from variable switches in the integral.

The term $e^{-N\left(\frac{1}{2}y^{T}Ly-\sum_{\lambda}\frac{N_{\lambda}}{N}\ln\cosh y_{\lambda}\right)}$
can be approximated by a Taylor series. For simplicity's sake, we call the coordinates $\left(x,y\right)$ instead of
$\left(y_{1},y_{2}\right)$. We calculate the
derivatives of $F$ up to order 2 in order to approximate $F$ around
$\left(0,0\right)$:
\begin{align*}
F\left(x,y\right)= & \frac{1}{2}\left[\left(\gamma_{1}x-\gamma_{2}y\right)^{2}+\frac{\alpha_{1}}{6}x^{4}+\frac{\alpha_{2}}{6}y^{4}\right].
\end{align*}
Recall that conditionally on $\left(x,y\right)\in\mathbb{R}^{2}$,
the spin variables are all independent, within each group also identically
distributed. Therefore,
\begin{equation}
E_{\left(x,y\right)}\left(\exp\left(it\left(\frac{\gamma_{1}}{\sqrt{\alpha_{1}N_{1}}}S_{1}-\frac{\gamma_{2}}{\sqrt{\alpha_{2}N_{2}}}S_{2}\right)\right)\right)=\frac{E_{x}\left(\exp\left(it\frac{\gamma_{1}}{\sqrt{\alpha_{1}N_{1}}}S_{1}\right)\right)}{E_{y}\left(\exp\left(it\frac{\gamma_{2}}{\sqrt{\alpha_{2}N_{2}}}S_{2}\right)\right)},\label{eq:cond_exp_asymp}
\end{equation}
and
\begin{align*}
E_{x}\left(\exp\left(it\frac{\gamma_{1}}{\sqrt{\alpha_{1}N_{1}}}S_{1}\right)\right) & =\exp\left(it\frac{\gamma_{1}x}{\sqrt{\alpha_{1}N_{1}}}N_{1}\right)E_{y_{1}}\left(\exp\left(it\frac{\gamma_{1}\left(X_{11}-x\right)}{\sqrt{\alpha_{1}N_{1}}}\right)\right)^{N_{1}}\\
 & \approx\exp\left(it\gamma_{1}x\sqrt{N}\right)\left(1-\left(1-x^{2}\right)\frac{\gamma_{1}^{2}t^{2}}{2\alpha_{1}N_{1}}+O\left(\frac{1}{N_{1}^{3/2}}\right)\right)^{N_{1}},
\end{align*}
where we used a Taylor series for the exponential function. This last
term is asymptotically equal to
\[
\exp\left(it\gamma_{1}x\sqrt{N}\right)\cdot\exp\left(-\left(1-x^{2}\right)\frac{\gamma_{1}^{2}t^{2}}{2\alpha_{1}}\right).
\]
We calculate an asymptotic expression similarly for group two and
obtain for the conditional expectation (\ref{eq:cond_exp_asymp})
\[
\exp\left(it\left(\gamma_{1}x-\gamma_{2}y\right)\sqrt{N}\right)\cdot\exp\left(-\left(\left(1-x^{2}\right)\frac{\gamma_{1}^{2}}{\alpha_{1}}+\left(1-y^{2}\right)\frac{\gamma_{1}^{2}}{\alpha_{1}}\right)\frac{t^{2}}{2}\right).
\]
When we apply Laplace's method to approximate the value of the integral,
the factor
\[
\exp\left(-\left(\left(1-x^{2}\right)\frac{\gamma_{1}^{2}}{\alpha_{1}}+\left(1-y^{2}\right)\frac{\gamma_{1}^{2}}{\alpha_{1}}\right)\frac{t^{2}}{2}\right)
\]
contributes only with its value at $\left(x,y\right)=\left(0,0\right)$.
Proceeding from the de Finetti representation of the characteristic
function $\varphi_{N}$ given above, we substitute
\begin{align*}
u & :=\left(\gamma_{1}x-\gamma_{2}y\right)\sqrt{N}\quad\textup{and}\quad v:=\left(\gamma_{1}x+\gamma_{2}y\right)N^{1/4},
\end{align*}
and the integral becomes (up to a multiplicative constant)
\begin{align*}
 & \quad\int_{\mathbb{R}^{2}}\exp\left(itu\right)\exp\left(-\left(\frac{\gamma_{1}^{2}}{\alpha_{1}}+\frac{\gamma_{2}^{2}}{\alpha_{2}}\right)\frac{t^{2}}{2}\right)e^{-\frac{1}{2}\left[u^{2}+\frac{\alpha_{1}}{2^{5}\cdot3\gamma_{1}^{2}}\left(\frac{u}{N^{1/4}}+v\right)^{4}+\frac{\alpha_{2}}{2^{5}\cdot3\gamma_{2}^{2}}\left(v-\frac{u}{N^{1/4}}\right)^{4}\right]}\mathrm{d}u\,\mathrm{d}v\\
 & \approx\exp\left(-\left(\frac{\gamma_{1}^{2}}{\alpha_{1}}+\frac{\gamma_{2}^{2}}{\alpha_{2}}\right)\frac{t^{2}}{2}\right)\int_{\mathbb{R}^{2}}\exp\left(itu\right)e^{-\frac{1}{2}\left[u^{2}+\frac{v^{4}}{2^{5}\cdot3}\left(\frac{\alpha_{1}}{\gamma_{1}^{2}}+\frac{\alpha_{2}}{\gamma_{2}^{2}}\right)\right]}\mathrm{d}u\,\mathrm{d}v
\end{align*}
by dominated convergence. Now note that the integral above can be
separated into two factors. The factor $\int_{\mathbb{R}}e^{-\frac{v^{4}}{2^{6}\cdot3}\left(\frac{\alpha_{1}}{\gamma_{1}^{2}}+\frac{\alpha_{2}}{\gamma_{2}^{2}}\right)}\mathrm{d}v$
cancels out with the corresponding term in the normalisation constant
$Z^{-1}$. The only term left in $Z^{-1}$ is $\int_{\mathbb{R}^{2}}e^{-\frac{1}{2}u^{2}}\mathrm{d}u$.
Thus we get
\begin{align*}
\varphi_{N}\left(t\right) & \approx\exp\left(-\left(\frac{\gamma_{1}^{2}}{\alpha_{1}}+\frac{\gamma_{2}^{2}}{\alpha_{2}}\right)\frac{t^{2}}{2}\right)\frac{1}{\sqrt{2\pi}}\int_{\mathbb{R}^{2}}\exp\left(itu\right)e^{-\frac{1}{2}u^{2}}\mathrm{d}u\\
 & =\exp\left(-\left(1+\frac{\gamma_{1}^{2}}{\alpha_{1}}+\frac{\gamma_{2}^{2}}{\alpha_{2}}\right)\frac{t^{2}}{2}\right).
\end{align*}
This concludes the proofs of the critical regime results.

\subsection{Low Temperature Regime}

The limit Theorem \ref{thm:LLN_het_low_temp}
for the magnetisations in the low temperature regime easily follows
from the correlations. The normalised sums are
\[
\boldsymbol{S}:=\left(\frac{S_{1}}{N_{1}},\ldots,\frac{S_{M}}{N_{M}}\right),
\]
and we define $M_{K_{1},\ldots,K_{M}}$ to be the moments of order
$\mathbf{K}=(K_{1},\ldots,K_{M})$ of this random vector, and set
$K:=\sum_{i=1}^{M}K_{i}$. These moments equal
\begin{align}
M_{K_{1},\ldots,K_{M}} & =\mathbb{E}\left(\left(\frac{S_{1}}{N_{1}}\right)^{K_{1}}\cdots\left(\frac{S_{M}}{N_{M}}\right)^{K_{M}}\right)\nonumber \\
 & =\frac{1}{N_{1}}\cdots\frac{1}{N_{M}}\sum_{\text{\ensuremath{\underbar{\ensuremath{\mathbf{r}}}}}}w_{\mathbf{K}}(\underbar{\ensuremath{\mathbf{r}}})\mathbb{E}\left(X_{\text{\ensuremath{\underbar{\ensuremath{\mathbf{r}}}}}}\right).\label{eq:moments_LLN_low}
\end{align}
The correlations $\mathbb{E}\left(X_{\text{\ensuremath{\underbar{\ensuremath{\mathbf{r}}}}}}\right)$
are constant with respect to $N$ for such $\underbar{\ensuremath{\mathbf{r}}}$
where each group $\lambda$ has a profile vector $\underline{r_{\lambda}}$
with $r_{\lambda1}=K_{\lambda}$. For all other types of profiles
$\underbar{\ensuremath{\mathbf{r}}}$, the summands converge to 0.

We therefore separate the sum in (\ref{eq:moments_LLN_low}) into
two summands:
\begin{align*}
A_{1} & =\frac{1}{N_{1}^{K_{1}}\cdots N_{M}^{K_{M}}}\sum_{\text{\ensuremath{\underbar{\ensuremath{\mathbf{r}}}}:\ensuremath{\forall\lambda:\underline{r_{\lambda}}\in\Pi_{K_{\lambda}}^{(K_{\lambda})}}}}w_{\mathbf{K}}(\underbar{\ensuremath{\mathbf{r}}})\mathbb{E}\left(X_{\text{\ensuremath{\underbar{\ensuremath{\mathbf{r}}}}}}\right),\\
A_{2} & =\frac{1}{N_{1}^{K_{1}}\cdots N_{M}^{K_{M}}}\sum_{\text{\ensuremath{\underbar{\ensuremath{\mathbf{r}}}}:\ensuremath{\exists\lambda:\underline{r_{\lambda}}\notin\Pi_{K_{\lambda}}^{(K_{\lambda})}}}}w_{\mathbf{K}}(\underbar{\ensuremath{\mathbf{r}}})\mathbb{E}\left(X_{\text{\ensuremath{\underbar{\ensuremath{\mathbf{r}}}}}}\right).
\end{align*}
We claim
\begin{prop}
The limit of $A_{2}$ as $N\rightarrow\infty$ is $0$.
\end{prop}

\begin{proof}
The proof is very similar to that of Proposition \ref{prop:Pi_plus_no_contrib}. We choose to
omit it.

\end{proof}
It follows from this proposition that
\begin{align*}
M_{K_{1},\ldots,K_{M}} & \approx A_{1}\approx m^{{\bf K}},
\end{align*}
where $m$ stands for $\left(m\left(\beta\right),\ldots,m\left(\beta\right)\right)$
if $J$ is homogeneous and $m:=\tanh\bar{m}$, the componentwise tanh of $\bar{m}$, if $J$ is heterogeneous
and $\bar{m}$ the minimum of $F$ in the positive orthant. This concludes
the proof of Theorems \ref{thm:LLN_hom_low_temp} and \ref{thm:LLN_het_low_temp}.

\subsubsection{\label{sec:Proof_Fluctuations}Proof of Theorem \ref{thm:cond_CLT}}

Now we turn to the proof of the conditional CLT presented in Theorem
\ref{thm:cond_CLT}. We prove the statement for heterogeneous coupling
matrices conditional on $S_{\nu}>0$ for each group $\nu$. The other
statements can be shown analogously. Let $\bar{m}$ and $m$ be defined
as in Theorem \ref{thm:cond_CLT}. We define the function $\tilde{F}$
such that its value at the global minimum is 0:
\begin{align*}
\tilde{F}(y) & :=F(y)-F\left(\bar{m}\right).
\end{align*}
We define the sequence of normalised random vectors
\[
\Sigma:=\left(\frac{1}{\sqrt{N_{1}}}\sum_{i_{1}=1}^{N_{1}}\left(X_{1i_{1}}-m_{1}\right),\ldots,\frac{1}{\sqrt{N_{M}}}\sum_{i_{M}=1}^{N_{M}}\left(X_{Mi_{M}}-m_{M}\right)\right),
\]
as well as the sequence of random variables
\begin{align*}
\chi^{+} & :=\begin{cases}
1, & \text{if for all groups }\nu\quad S_{\nu}>0,\\
0, & \text{otherwise.}
\end{cases}
\end{align*}
We have to calculate the moments of all orders $\mathbf{K}=\left(K_{1},\ldots,K_{M}\right)\in\mathbb{N}_{0}^{M}$
of the random vectors $\left(\Sigma\chi^{+}\right)^{\mathbf{K}}$.
Note that $\left(\Sigma\chi^{+}\right)^{\mathbf{K}}=\Sigma^{\mathbf{K}}\chi^{+}$
holds for all $\mathbf{K}\neq\mathbf{0}$. We define the sequence
of integrals
\[
\mathcal{Z}_{\mathbf{K}}:=\int_{\IR^{M}}e^{-N\tilde{F}(y)}E_{y}^{\otimes N}\left(\Sigma^{\mathbf{K}}\chi^{+}\right)\text{d}y.
\]
The moments can thus be calculated by the formula
\[
\mathbb{E}\left(\left(\Sigma\chi^{+}\right)^{\mathbf{K}}\right)=\frac{\mathcal{Z}_{\mathbf{K}}}{\mathcal{Z}_{\mathbf{0}}}.
\]
We calculate $\mathcal{Z}_{\mathbf{K}}$ asymptotically by dividing
it into three separate parts and showing that only one of these contributes
asymptotically to the value of $\mathcal{Z}_{\mathbf{K}}$. For this
purpose, we define these subsets of $\IR^{M}$:
\begin{align*}
A_{-} & :=\prod_{\lambda=1}^{M}\left(-\infty,-\frac{\bar{m}_{\lambda}}{2}\right],\quad A_{+}:=\prod_{\lambda=1}^{M}\left[\frac{\bar{m}_{\lambda}}{2},\infty\right),\quad A_{0}:=\IR^{M}\backslash\left(A_{-}\cup A_{+}\right),
\end{align*}
as well as the integrals over these subsets:
\begin{align*}
\mathcal{Z}_{\mathbf{K}}^{-} & :=\int_{A_{-}}e^{-N\tilde{F}(y)}E_{y}^{\otimes N}\left(\Sigma^{\mathbf{K}}\chi^{+}\right)\text{d}y,\\
\mathcal{Z}_{\mathbf{K}}^{0} & :=\int_{A_{0}}e^{-N\tilde{F}(y)}E_{y}^{\otimes N}\left(\Sigma^{\mathbf{K}}\chi^{+}\right)\text{d}y,\\
\mathcal{Z}_{\mathbf{K}}^{+} & :=\int_{A_{+}}e^{-N\tilde{F}(y)}E_{y}^{\otimes N}\left(\Sigma^{\mathbf{K}}\chi^{+}\right)\text{d}y.
\end{align*}
We state a result concerning the speed of convergence of sums of independent
random variables which we will use in this proof.
\begin{lem}
\label{lem:iiid_fast_convergence}Let $\left(X_{n}\right)$ be a sequence
of independent random variables on a probability space $(\Omega,\mathcal{F},P)$.
Assume that for each $m\in\mathbb{N}_{0}$ there is a constant $T_{m}$
such that
\[
\sup_{n\in\mathbb{N}}E\left(X_{n}^{m}\right)\leq T_{m}<\infty.
\]
Then, for each $V\in\IN$, there is a constant $C_{V}$ such that
for any $a>0$
\[
P\left(\left|\frac{1}{N}\sum_{n=1}^{N}\left(X_{n}-E(X_{n})\right)\right|>a\right)\leq\frac{C_{V}}{a^{2V}N^{V}}.
\]
\end{lem}

\begin{proof}
This is Theorem 3.24 in \cite{MM}.
\end{proof}
We first show
\begin{lem}
\label{lem:Z_-1}$\mathcal{Z}_{\mathbf{K}}^{-}$ converges to $0$
faster than any power of $N$.
\end{lem}

\begin{proof}
Let $Q\in\mathbb{N}$. We prove that $E_{y}^{\otimes N}\left(\Sigma^{\mathbf{K}}\chi^{+}\right)$
goes to 0 faster than $1/N^{Q}$.

Given a $y\in A_{-}$, the Cauchy-Schwarz inequality states that
\[
E_{y}^{\otimes N}\left(\Sigma^{\mathbf{K}}\chi^{+}\right)\leq E_{y}^{\otimes N}\left(\Sigma^{2\mathbf{K}}\right)^{\frac{1}{2}}P_{y}^{\otimes N}\left(\chi^{+}=1\right)^{\frac{1}{2}}.
\]
We calculate upper bounds for both expressions on the right hand side
of the above inequality. Due to the conditional independence of the
random variables $X_{\lambda i}$, we have
\begin{align}
E_{y}^{\otimes N}\left(\Sigma^{2\mathbf{K}}\right) & =E_{y}^{\otimes N}\left(\left(\frac{1}{\sqrt{N_{1}}}\sum_{i_{1}=1}^{N_{1}}\left(X_{1i_{1}}-m_{1}\right)\right)^{2K_{1}}\right)\cdots E_{y}^{\otimes N}\left(\left(\frac{1}{\sqrt{N_{M}}}\sum_{i_{M}=1}^{N_{M}}\left(X_{Mi_{M}}-m_{M}\right)\right)^{2K_{M}}\right).\label{eq:cond_expec_Sigma_2K}
\end{align}
The conditionally i.i.d.\! random variables $X_{\nu i}$ for fixed $\nu$ and $i=1,\ldots,N_\nu$ each have a
conditional expectation of $y_{\nu}\leq-\frac{m_{\nu}}{2}$, so a
lower bound on the distance between $E_{y}\left(X_{\nu i}\right)$
and $m_{\nu}$ is given by
\begin{align*}
\left|E_{y}\left(X_{\nu i}\right)-m_{\nu}\right| & =\left|y_{\nu}-m_{\nu}\right|\geq\frac{3}{2}m_{\nu}>0.
\end{align*}
Thus, we cannot expect these expectations to converge to 0. However, we have the following upper bound for each factor in (\ref{eq:cond_expec_Sigma_2K}):
\begin{align*}
\left(\frac{1}{\sqrt{N_{\nu}}}\sum_{i_{\nu}=1}^{N_{\nu}}\left(X_{\nu i_{\nu}}-m_{\nu}\right)\right)^{2K_{\nu}} & \leq\frac{1}{N_{\nu}^{K_{\nu}}}\left(\sum_{i_{\nu}=1}^{N_{\nu}}\left(\left|X_{\nu i_{\nu}}\right|+\left|m_{\nu}\right|\right)\right)^{2K_{\nu}}\\
 & \leq\frac{1}{N_{\nu}^{K_{\nu}}}\left(2N_{\nu}\right)^{2K_{\nu}}=2^{2K_{\nu}}N_{\nu}^{K_{\nu}}.
\end{align*}
Thus, we obtain the upper bound
\begin{align}
E_{y}^{\otimes N}\left(\Sigma^{2\mathbf{K}}\right)^{\frac{1}{2}} & \leq\left(\prod_{\nu=1}^{M}2^{2K_{\nu}}N_{\nu}^{K_{\nu}}\right)^{\frac{1}{2}}\approx 2^{K}\prod_{\nu=1}^{M}\alpha_{\nu}^{\frac{K_{\nu}}{2}}N^{\frac{K}{2}},\label{eq:upper_Sigma_2K}
\end{align}
where we set for the rest of this section $K:=\sum_{\nu=1}^{M}K_{\nu}.$
We use Lemma \ref{lem:iiid_fast_convergence} in the third step below
to obtain an upper bound for $P_{y}^{\otimes N}\left(\chi^{+}=1\right)$:
\begin{align}
P_{y}^{\otimes N}\left(\chi^{+}=1\right) & =\prod_{\nu=1}^{M}P_{y}^{\otimes N}\left(\frac{1}{N_{\nu}}\sum_{i_{\nu}=1}^{N_{\nu}}\left(X_{\nu i_{\nu}}-y_{\nu}\right)+y_{\nu}>0\right)\nonumber \\
 & \leq\prod_{\nu=1}^{M}P_{y}^{\otimes N}\left(\left|\frac{1}{N_{\nu}}\sum_{i_{\nu}=1}^{N_{\nu}}\left(X_{\nu i_{\nu}}-t_{\nu}\right)\right|>-y_{\nu}\right)\nonumber \\
 & \leq\prod_{\nu=1}^{M}\frac{C_{2V}}{\left(-y_{\nu}\right)^{4V}N_{\nu}^{2V}}\leq\prod_{\nu=1}^{M}\frac{C_{2V}}{\left(-\frac{m_{\nu}}{2}\right)^{4V}N_{\nu}^{2V}}\nonumber \\
 & =\prod_{\nu=1}^{M}\frac{2^{4V}C_{2V}}{\left(m_{\nu}\right)^{4V}\alpha_{\nu}^{2V}}\cdot\frac{1}{N^{2MV}}.\label{eq:upper_chi_plus}
\end{align}
For a given $\mathbf{K}$, choose
$V(\mathbf{K}):=\frac{K}{2}+Q+1$. Then, using the upper bounds in
(\ref{eq:upper_Sigma_2K}) and (\ref{eq:upper_chi_plus}), we obtain
\begin{align*}
E_{y}^{\otimes N}\left(\Sigma^{\mathbf{K}}\chi^{+}\right) & \leq cN^{\frac{K}{2}}\frac{1}{N^{MV(\mathbf{K})}}=cN^{-(M-1)\frac{K}{2}-M(Q+1)},
\end{align*}
which goes to 0 faster than $1/N^{Q}$ for any number of groups $M\in\mathbb{N}$. Set $Q':=(M-1)\frac{K}{2}+M(Q+1)$.
The constant $c$ depends only on $\mathbf{K}$ and $Q$ but not on
$N$.

Thus, we see that for any $Q\in\mathbb{N}$
\begin{align*}
\mathcal{Z}_{\mathbf{K}}^{-} & \leq c\frac{1}{N^{Q'}} \int_{A_{-}}e^{-N\tilde{F}(y)}\text{d}y\rightarrow0\text{ as }N\rightarrow\infty.
\end{align*}
Note that the sequence of integrals above converges. This concludes the proof of the lemma.
\end{proof}
We next show
\begin{lem}
$\mathcal{Z}_{\mathbf{K}}^{0}$ converges to $0$ exponentially fast
in $N$.
\end{lem}

\begin{proof}
Since $\tilde{F}$ has exactly two global minima at $\pm\bar{m}\notin A_{0}$,
there is a $\delta>0$ such that $\tilde{F}\left(\pm\bar{m}\right)=0<\delta\leq\tilde{F}(y)$
for all $y\in A_{0}$. Similarly to the proof of the last lemma, the
expectation $E_{y}^{\otimes N}\left(\Sigma^{2\mathbf{K}}\right)$
does not converge to zero for $y\in A_{0}$, as the distance between
$E_{y}\left(X_{\nu i}\right)$ and $m_{\nu}$ is at least
\begin{align*}
\left|E_{y}\left(X_{\nu i}\right)-m_{\nu}\right| & =\left|y_{\nu}-m_{\nu}\right|\geq\frac{m_{\nu}}{2}>0.
\end{align*}
However, we can once again use the upper bound for $E_{y}^{\otimes N}\left(\Sigma^{2\mathbf{K}}\right)^{\frac{1}{2}}$
given by (\ref{eq:upper_Sigma_2K}). We calculate
\begin{align*}
\mathcal{Z}_{\mathbf{K}}^{0} & \leq\int_{A_{0}}e^{-N\delta}E_{y}^{\otimes N}\left(\Sigma^{2\mathbf{K}}\right)^{\frac{1}{2}}P_{y}^{\otimes N}\left(\chi^{+}=1\right)^{\frac{1}{2}}\text{d}y\\
 & \leq2^{K}\prod_{\nu=1}^{M}\alpha_{\nu}^{\frac{K_{\nu}}{2}}N^{\frac{K}{2}}e^{-N\delta}\int_{A_{0}}\text{d}y.
\end{align*}
This last expression converges to 0 exponentially fast as $N\rightarrow\infty$.
\end{proof}
We turn our attention to $\mathcal{Z}_{\mathbf{K}}^{+}$. Using
\[
E_{y}^{\otimes N}\left(\Sigma^{\mathbf{K}}\chi^{+}\right)=E_{y}^{\otimes N}\left(\Sigma^{\mathbf{K}}\right)-E_{y}^{\otimes N}\left(\Sigma^{\mathbf{K}}\left(1-\chi^{+}\right)\right),
\]
we divide $\mathcal{Z}_{\mathbf{K}}^{+}$ into two parts
\begin{align}
\mathcal{Z}_{\mathbf{K}}^{+} & =\int_{A_{+}}e^{-N\tilde{F}(y)}E_{y}^{\otimes N}\left(\Sigma^{\mathbf{K}}\right)\text{d}y\label{eq:Z_1_contrib}\\
 & \quad-\int_{A_{+}}e^{-N\tilde{F}(y)}E_{y}^{\otimes N}\left(\Sigma^{\mathbf{K}}\left(1-\chi^{+}\right)\right)\text{d}y.\label{eq:Z_1_no_contrib}
\end{align}
By the same reasoning as in Lemma \ref{lem:Z_-1}, the expression
(\ref{eq:Z_1_no_contrib}) goes to 0 faster than any power of $N$.
We centre our attention on the expression (\ref{eq:Z_1_contrib}).
We calculate the conditional expectation
\begin{align*}
E_{y}^{\otimes N}\left(\Sigma^{\mathbf{K}}\right) & =E_{y}^{\otimes N}\left(\prod_{\nu=1}^{M}\left(\frac{\sum_{i_{\nu}=1}^{N_{\nu}}\left(X_{\nu i_{\nu}}-m_{\nu}\right)}{\sqrt{N_{\nu}}}\right)^{K_{\nu}}\right)
 =\prod_{\nu=1}^{M}\frac{E_{y}^{\otimes N}\left(\left(\sum_{i_{\nu}=1}^{N_{\nu}}Y_{\nu i_{\nu}}\right)^{K_{\nu}}\right)}{N_{\nu}^{\frac{K_{\nu}}{2}}},
\end{align*}
where we set for each $\nu$ and each $i_{\nu}$ $Y_{\nu i_{\nu}}:=X_{\nu i_{\nu}}-m_{\nu}$.

As usual in these types of proofs, we now have to see what profiles
of multiindices contribute asymptotically to the moments:
\begin{align*}
E_{y}^{\otimes N}\left(\left(\sum_{i_{\nu}=1}^{N_{\nu}}Y_{\nu i_{\nu}}\right)^{K_{\nu}}\right) & =\sum_{j_{1},\ldots,j_{K_{\nu}}=1}^{N_{\nu}}E_{y}^{\otimes N}\left(Y_{\nu j_{1}}\cdots Y_{\nu j_{K_{\nu}}}\right)
 =\sum_{\underline{r}}w_{K_{\nu}}\left(\underline{r}\right)E_{y}^{\otimes N}\left(Y_{\text{\ensuremath{\underline{r}}}}\right).
\end{align*}
We have by Proposition \ref{thm:comb-coeff-multiindex}
\[
w_{K_{\nu}}(\underline{r})\approx\frac{N_{\nu}^{\sum_{\ell}r_{\ell}}}{r_{1}!r_{2}!\cdots r_{K_{\nu}}!}\frac{K_{\nu}!}{2!^{r_{2}}\cdots K_{\nu}!^{r_{K_{\nu}}}}.
\]
We first note that for a group $\nu$ and $\text{\ensuremath{\underline{r}}}\in\Pi^{0(K_{\nu})}$,
i.e.\! a profile vector with
\begin{align*}
\underline{r} & =\left(r_{1},\frac{K_{\nu}-r_{1}}{2},0,\ldots,0\right),
\end{align*}
and $r_{1}$ and $K_{\nu}$ have the same parity, the expression becomes
\begin{equation}
w_{K_{\nu}}(\underline{r})\approx\frac{N_{\nu}^{\frac{K_{\nu}+r_{1}}{2}}}{r_{1}!\left(\frac{K_{\nu}-r_{1}}{2}\right)!}\frac{K_{\nu}!}{2^{\frac{K_{\nu}-r_{1}}{2}}}.\label{eq:multipl_Pi_0}
\end{equation}
Due to the conditional independence of the $Y_{\nu i_{\nu}}$, we
have
\begin{align*}
E_{y}^{\otimes N}\left(Y_{\text{\ensuremath{\underline{r}}}}\right) & =E_{y}^{\otimes N}\left(Y_{\nu1}\right)^{r_{1}}E_{y}^{\otimes N}\left(Y_{\nu1}^{2}\right)^{\frac{K_{\nu}-r_{1}}{2}}\\
 & =\left(\tanh y_{\nu}-m_{\nu}\right)^{r_{1}}\left(1-2 \tanh y_{\nu} \cdot m_{\nu}+m_{\nu}^{2}\right)^{\frac{K_{\nu}-r_{1}}{2}},
\end{align*}
Let $\underbar{\ensuremath{\mathbf{r}}}:=\left(\underline{r_{1}},\underline{r_{2}},\ldots,\underline{r_{M}}\right)$
be a profile vector for all groups such that for each group $\nu$ we have $\underline{r_{\nu}}=\left(r_{\nu1},\ldots,r_{\nu K_{\nu}}\right)\in\Pi^{0(K_{\nu})}$.
Using Laplace's method in a similar setting as in Proposition \ref{prop:Laplace_multivariate},
we obtain
\begin{align*}
\mathcal{E}\left(\underbar{\ensuremath{\mathbf{r}}}\right) & :=\int_{A_{+}}e^{-N\tilde{F}(t)}E_{y}^{\otimes N}\left(Y_{\text{\ensuremath{\underbar{\ensuremath{\mathbf{r}}}}}}\right)\text{d}y\\
 & =\int_{A_{+}}e^{-N\tilde{F}(t)}\prod_{\nu=1}^{M}\left(\tanh y_{\nu}-m_{\nu}\right)^{r_{\nu1}}\left(1-2 \tanh y_{\nu} \cdot m_{\nu}+m_{\nu}^{2}\right)^{\frac{K_{\nu}-r_{\nu1}}{2}}\text{d}y\\
 & \approx\left(2\pi\right)^{\frac{M}{2}}\sqrt{\det H^{-1}\left(\bar{m}\right)}\left(\prod_{\nu=1}^{M}\left(1-m_{\nu}^{2}\right)^{\frac{K_{\nu}-r_{\nu1}}{2}}\frac{1}{N^{\frac{r_{\nu1}}{2}}}\right)\frac{1}{N^{\frac{1}{2}}}m_{r_{11},\ldots,r_{M1}}\left(H^{-1}\left(\bar{m}\right)\right),
\end{align*}
where $H^{-1}\left(\bar{m}\right)$ in the last line stands for the
inverse of the Hessian matrix of $\tilde{F}$ in the points $\pm\bar{m}$.

We need to normalise $\mathcal{E}\left(\underbar{\ensuremath{\mathbf{r}}}\right)$
by dividing the above expression by
\begin{align*}
\mathcal{E}\left(\underbar{\ensuremath{\mathbf{0}}}\right)  :=\int_{A_{+}}e^{-N\tilde{F}(t)}\text{d}y
  \approx\left(2\pi\right)^{\frac{M}{2}}\sqrt{\det H^{-1}}\frac{1}{N^{\frac{1}{2}}}m_{0,\ldots,0}\left(H^{-1}\right).
\end{align*}
Then we divide these two expressions and obtain
\begin{align}
\frac{\mathcal{E}\left(\underbar{\ensuremath{\mathbf{r}}}\right)}{\mathcal{E}\left(\underbar{\ensuremath{\mathbf{0}}}\right)} & =\left(\prod_{\nu=1}^{M}\left(1-m_{\nu}^{2}\right)^{\frac{K_{\nu}-r_{\nu1}}{2}}\frac{1}{N^{\frac{r_{\nu1}}{2}}}\right)m_{r_{11},\ldots,r_{M1}}\left(H^{-1}\right).\label{eq:expec_cond_moments}
\end{align}
Suppose there is at least one group $\nu$ such that
\[
\underline{r_{\nu}}=\left(r_{\nu1},\ldots,r_{\nu K_{\nu}}\right)\in\Pi^{+(K_{\nu})},
\]
i.e.\! there is some index that
repeats at least three times. By the same reasoning as in Proposition
\ref{prop:Pi_plus_no_contrib}, we conclude that summands corresponding to these profile vectors
do not contribute asymptotically, as they converge to 0, contrary
to those summands where $\underline{r_{\nu}} \in\Pi^{0(K_{\nu})}$
for each group.

To calculate the asymptotic moments, we collect the constants from
the expressions (\ref{eq:multipl_Pi_0}) and (\ref{eq:expec_cond_moments})
and sum over all profile vectors where each group $\nu$ belongs to $\Pi^{0(K_{\nu})}$:
\begin{align}
M_{\mathbf{K}}\left(\Sigma\chi^{+}\right) & \approx\sum_{k_{1}=0}^{K_{1}}\cdots\sum_{k_{M}=0}^{K_{M}}\prod_{\lambda=1}^{M}\frac{\alpha_{\lambda}^{\frac{k_{\lambda}}{2}}}{k_{\lambda}!\left(\frac{K_{\lambda}-k_{\lambda}}{2}\right)!}\frac{K_{\lambda}!}{2^{\frac{K_{\lambda}-k_{\lambda}}{2}}}\left(1-m_{\lambda}^{2}\right)^{\frac{K_{\lambda}-k_{\lambda}}{2}}m_{k_{1},\ldots,k_{M}}\left(H^{-1}\right)\nonumber \\
 & =\prod_{\nu=1}^{M}\left(1-m_{\nu}^{2}\right)^{\frac{K_{\nu}}{2}}\sum_{k_{1}=0}^{K_{1}}\cdots\sum_{k_{M}=0}^{K_{M}}\prod_{\lambda=1}^{M}\frac{\left(\frac{\alpha_{\lambda}}{1-m_{\nu}^{2}}\right)^{\frac{k_{\lambda}}{2}}}{k_{\lambda}!\left(\frac{K_{\lambda}-k_{\lambda}}{2}\right)!}\frac{K_{\lambda}!}{2^{\frac{K_{\lambda}-k_{\lambda}}{2}}}m_{k_{1},\ldots,k_{M}}\left(H^{-1}\right),\label{eq:moment_Cond_A_1}
\end{align}
where each sum is over those $k_{\lambda}$ which have the same parity
as $K_{\lambda}$.

The second factor above,
\[
\sum_{k_{1}=0}^{K_{1}}\cdots\sum_{k_{M}=0}^{K_{M}}\prod_{\lambda=1}^{M}\frac{\left(\frac{\alpha_{\lambda}}{1-m_{\nu}^{2}}\right)^{\frac{k_{\lambda}}{2}}}{k_{\lambda}!\left(\frac{K_{\lambda}-k_{\lambda}}{2}\right)!}\frac{K_{\lambda}!}{2^{\frac{K_{\lambda}-k_{\lambda}}{2}}}m_{k_{1},\ldots,k_{M}}\left(H^{-1}\right),
\]
has the same structure as the moment given in (\ref{eq:S_moments_A_1}).
By the proof of the CLT \ref{thm:CLT}, these are the moments of a centred multivariate
normal distribution. Since in (\ref{eq:moment_Cond_A_1}) there is
another factor present, namely
\[
\prod_{\nu=1}^{M}\left(1-m_{\nu}^{2}\right)^{\frac{K_{\nu}}{2}},
\]
we need the following
\begin{lem}
Let $\ell_{1},\ldots,\ell_{n}\in\mathbb{N}_{0}$. Let $m_{\ell_{1},\ldots,\ell_{n}}:=m_{\ell_{1},\ldots,\ell_{n}}(C)$
be the moment of the centred multivariate normal distribution with
covariance matrix $C=\left(c_{ij}\right)_{i,j=1,\ldots,n}$. Let $s_{1},\ldots,s_{n}$
be positive numbers. Then
\begin{equation*}
s_{1}^{\frac{\ell_{1}}{2}}\cdots s_{n}^{\frac{\ell_{n}}{2}}m_{\ell_{1},\ldots,\ell_{n}}
\end{equation*}
is the moment of order $\left(\ell_{1},\ldots,\ell_{n}\right)$ of
the centred multivariate distribution with covariance matrix
\[
D:=\left(d_{ij}\right)_{i,j=1,\ldots,n}:=\left(\sqrt{s_{i}s_{j}}c_{ij}\right)_{i,j=1,\ldots,n}.
\]
\end{lem}

\begin{proof}
We omit the proof which consists of an application of Isserlis's Theorem.
\end{proof}
This lemma shows that the expression (\ref{eq:moment_Cond_A_1}) is
the moment of order $\mathbf{K}$ of a centred multivariate normal
distribution.

To calculate the covariance matrix $E$ as a function of the Hessian
matrix $H$ of $\tilde{F}$ at $\pm\bar{m}$, we use the formula
(\ref{eq:moment_Cond_A_1}). The diagonal entries of $E$ are
\begin{align*}
e_{\lambda\lambda} & =\left(1-m_{\lambda}^{2}\right)\sum_{k_{\lambda}=0,2}\frac{\left(\frac{\alpha_{\lambda}}{1-\left(m_{\nu}^{*}\right)^{2}}\right)^{\frac{k_{\lambda}}{2}}}{k_{\lambda}!\left(\frac{2-k_{\lambda}}{2}\right)!}\frac{2!}{2^{\frac{2-k_{\lambda}}{2}}}m_{0,\ldots,0,k_{\lambda},0,\ldots,0}\left(H^{-1}\right)\\
 & =1-m_{\lambda}^{2}+\alpha_{\lambda}\left(H^{-1}\right)_{\lambda\lambda}.
\end{align*}
The off-diagonal entries are given by
\begin{align*}
e_{\lambda\mu} & =\sqrt{1-m_{\lambda}^{2}}\sqrt{1-m_{\mu}^{2}}\sum_{k_{\lambda}=1}^{1}\sum_{k_{\mu}=1}^{1}\frac{\left(\frac{\alpha_{\lambda}}{1-\left(m_{\nu}^{*}\right)^{2}}\right)^{\frac{k_{\lambda}}{2}}}{k_{\lambda}!\left(\frac{1-k_{\lambda}}{2}\right)!}\frac{1!}{2^{\frac{1-k_{\lambda}}{2}}}\frac{\left(\frac{\alpha_{\mu}}{1-\left(m_{\nu}^{*}\right)^{2}}\right)^{\frac{k_{\mu}}{2}}}{k_{\mu}!\left(\frac{1-k_{\mu}}{2}\right)!}\frac{1!}{2^{\frac{1-k_{\mu}}{2}}}\cdot\\
 & \quad\cdot m_{0,\ldots,0,k_{\lambda},0,\ldots,0,k_{\mu},0,\ldots,0}\left(H^{-1}\right)\\
 & =\sqrt{\alpha_{\lambda}\alpha_{\mu}}\left(H^{-1}\right)_{\lambda\mu}.
\end{align*}
We can thus express the covariance matrix as
\[
E=\text{diag}\left(1-m_{\lambda}^{2}\right)+\sqrt{\boldsymbol{\alpha}}H^{-1}\sqrt{\boldsymbol{\alpha}}.
\]

\section{\label{sec:SLLN}Strong Laws of Large Numbers}

In this section, we prove the statements in Remarks \ref{rem:LLN_high}
and \ref{rem:LLN_crit}. It is well known that a sequence of random
variables which converges in distribution to a constant also converges
in probability. Thus the Weak Law of Large Numbers holds for the
LLNs \ref{thm:LLN_high_temp} and \ref{thm:LLN_crit_reg}.

We next state a proposition that will allow us to show almost sure
convergence.
\begin{prop}
\label{prop:summable_events}Let $(X_{n})$ be a sequence of real
random variables defined on the probability space $(\Omega,\mathcal{A},P)$.
Set $A_{m,k}:=\left\{ |X_{m}|>\frac{1}{k}\right\} $, $A:=\{\lim_{n\rightarrow\infty}X_{n}=0\}.$
If $\sum_{m=1}^{\infty}P(A_{m,k})<\infty$ for all $k\in\mathbb{N}$,
then $P(A)=1$.
\end{prop}

\begin{proof}
This is Proposition 3.33 in \cite{MM}.
\end{proof}
Now we show that the normalised sums $\frac{S_{\lambda}}{N_{\lambda}}$
converge to 0 faster than any power of $N$ when we are either in
the high temperature or the critical regime.
\begin{thm}
Assume we are in the high temperature or the critical regime, and
let $K\in\mathbb{N}$ and $a>0$. Then there is a constant $d_{K}$
such that
\[
\mathbb{P}\left(\left\Vert \left(\frac{S_{1}}{N_{1}},\ldots,\frac{S_{M}}{N_{M}}\right)\right\Vert >a\right)\leq\frac{d_{K}}{a^{4K}N^{K}},
\]
where $\left\Vert \cdot\right\Vert $ stands for the max norm on $\mathbb{R}^{M}$.
\end{thm}

\begin{proof}
Let $K\in\mathbb{N}$. We show for all $\lambda$:
\[
\mathbb{P}\left(\left|\frac{S_{\lambda}}{N_{\lambda}}\right|>a\right)\leq\frac{d_{K}}{a^{4K}N^{K}}.
\]
We have
\begin{align*}
\mathbb{P}\left(\left|\frac{S_{\lambda}}{N_{\lambda}}\right|>a\right) & \leq\frac{1}{a^{4K}N_{\lambda}^{K}}\mathbb{E}\left(\left|\frac{S_{\lambda}}{N_{\lambda}^{\frac{3}{4}}}\right|^{4K}\right),
\end{align*}
where we applied Markov's inequality.

If we are in the high temperature regime, then $\mathbb{E}\left(\left|\frac{S_{\lambda}}{N_{\lambda}^{\frac{3}{4}}}\right|^{4K}\right)$
converges to 0. In the critical regime, we have $\mathbb{E}\left(\left|\frac{S_{\lambda}}{N_{\lambda}^{\frac{3}{4}}}\right|^{4K}\right)\approx b_{4K},$
where $b_{4K}=c_{0,\ldots,0,4K,0,\ldots,0}\alpha_{\lambda}^{K}$ from
(\ref{eq:constants_crit}) with the entry $4K$ at the position $\lambda$
and 0 everywhere else. We are done once we set $d_{K}$ equal to the maximum of these $b_{4K}$ over all $\lambda$.
\end{proof}
Now we can prove that the Strong Law of Large Numbers holds in the
high temperature and critical regimes.
\begin{thm}
Assume we are in the high temperature or the critical regime. Then
\[
\lim_{N\rightarrow\infty}\left(\frac{S_{1}}{N_{1}},\ldots,\frac{S_{M}}{N_{M}}\right)=(0,\ldots,0)\quad \text{a.s.}
\]
\end{thm}

\begin{proof}
Let $\lambda\in\{1,\ldots,M\}$, and let for all $N,k\in\mathbb{N}$
\[
A_{N,k}:=\left\{ \left|\frac{S_{\lambda}}{N_{\lambda}}\right|>\frac{1}{k}\right\} .
\]
By the previous theorem, we have
\[
\mathbb{P}\left(A_{N,k}\right)\leq\frac{d_{2}k^{8}}{N^{2}}
\]
for some constant $d_{2}$. Then $\sum_{N=1}^{\infty}\mathbb{P}\left(A_{N,k}\right)<\infty$
holds and we can apply Proposition \ref{prop:summable_events} to
conclude that
\[
\mathbb{P}\left(\lim_{N\rightarrow\infty}\frac{S_{\lambda}}{N_{\lambda}}=0\right)=1.
\]
\end{proof}

\section*{Declarations Section}

\subsection*{Conflict of Interests}
Not applicable.

\subsection*{Availability of Data and Materials}
Not applicable.


\begin{thebibliography}{10}
\bibitem{BRS}Berthet, Quentin; Rigollet, Philippe; Srivastava,Piyush:
Exact recovery in the Ising blockmodel, Ann. Statist. 47 (4) 1805
- 1834, August 2019

\bibitem{BD}Brock, William A.; Durlauf, Steven N.: Discrete Choice
with Social Interactions, Review of Economic Studies, Oxford University
Press, vol. 68(2), pages 235-260. 2001

\bibitem{Co2014}Collet, Francesca: Macroscopic limit of a bipartite
Curie-Weiss model: a dynamical approach. J. Stat. Phys., 157(6): 1301-{}-1319,
2014

\bibitem{CG2008}Contucci, Pierluigi; Gallo, Ignacio: Bipartite Mean
Field Spin Systems. Existence and Solution, Math. Phys. Elec. Jou.
Vol 14, N.1, 1-22 (2008)

\bibitem{CGh}Contucci, Pierluigi and Ghirlanda, S.: Modelling Society
with Statistical Mechanics: an Application to Cultural Contact and
Immigration. Quality and Quantity, 41, 569-578 (2007)

\bibitem{Ellis} Ellis, Richard: Entropy, large deviations, and statistical
mechanics, Whiley (1985)

\bibitem{EllisN} Ellis, Richard; Newman, Charles: Limit theorems for sums of dependent random variables occurring in statistical mechanics, Z. Wahrsch. Verw. Gebiete 44, 117--139 (1978) 

\bibitem{FM} Fedele, Micaela: Rescaled Magnetization for Critical
Bipartite Mean-Fields Models, J. Stat. Phys. 155:223--226 (2014)

\bibitem{FC} Fedele, Micaela; Contucci, Pierluigi: Scaling Limits
for Multi-species Statistical Mechanics Mean-Field Models, J. Stat.
Phys. 144:1186--1205 (2011)

\bibitem{F=0000F6}F\"ollmer, Hans: Random economies with many interacting
agents, Journal of Mathematical Economics, Volume 1, Issue 1, 1974,
51-62

\bibitem{GBC}Gallo, Ignacio; Barra, Adriano; Contucci, Pierluigi;
Parameter Evaluation of a Simple Mean-Field Model of Social Interaction,
Math. Models Methods Appl. Sci., 19 (suppl.), pp. 1427-1439 (2009)

\bibitem{Husimi}Husimi, K.: Statistical Mechanics of Condensation,
Proceedings of the International Conference of Theoretical Physics,
pp. 531-533, Science Council of Japan, Tokyo (1953)

\bibitem{Isserlis}Isserlis, Leon: On a Formula for the Product-Moment
Coefficient of any Order of a Normal Frequency Distribution in any
Number of Variables, Biometrika, Vol. 12, No. 1/2, pp. 134-139 (1918)

\bibitem{Kac}Kac, M.: Mathematical Mechanisms of Phase Transitions,
in Statistical Physics: Phase Transitions and Superfluidity, Vol.
1, pp. 241-305, Brandeis University Summer Institute in Theoretical
Physics (1968)

\bibitem{MM}Kirsch, Werner: A Survey on the Method of Moments, available
from http://www.fernuni-hagen.de/stochastik/

\bibitem{WK-M} Kirsch, Werner: The Curie-Weiss model -- an approach
using moments. M\"{u}nster J. of Math. 13, 205--218 (2020)

\bibitem{WK-HO} Kirsch, Werner: On Penrose's Square-root Law and Beyond. Homo Oeconomicus 24, 357--380 (2007)

\bibitem{KL} Kirsch, W., Langner J.: The Fate of the Square Root Law for Correlated Voting.\\
in: Fara, R et al: Voting Power ans Procedures, Springer 2014

\bibitem{KT1}Kirsch, W., Toth, G. Two Groups in a Curie--Weiss Model.
Math Phys Anal Geom 23, 17 (2020). https://doi.org/10.1007/s11040-020-09343-5

\bibitem{KT2}Kirsch, W., Toth, G. Two Groups in a Curie--Weiss Model
with Heterogeneous Coupling. J Theor Probab 33, 2001--2026 (2020).
https://doi.org/10.1007/s10959-019-00933-w

\bibitem{KT_opt_weights_CW}Kirsch, W., Toth, G.: Optimal Weights
in a Two-Tier Voting System with Mean-Field Voters, arXiv:2111.08636
(2021)

\bibitem{KLSS}Kn\"opfel, H., L\"owe, M., Schubert, K. et al. Fluctuation
Results for General Block Spin Ising Models. J Stat Phys 178, 1175--1200
(2020). https://doi.org/10.1007/s10955-020-02489-0

\bibitem{LS}L\"owe, Matthias; Schubert, Kristina: Fluctuations for
Block Spin Ising Models, Electron. Commun. Probab., Volume 23 paper
no. 53 (2018)

\bibitem{LSV}Matthias L\"owe, Kristina Schubert, Franck Vermet, Multi-group
binary choice with social interaction and a random communication structure---A
random graph approach, Physica A: Statistical Mechanics and its Applications,
Volume 556, 2020

\bibitem{OEA}Opoku, A.A., Edusei, K.O. \& Ansah, R.K. A Conditional
Curie--Weiss Model for Stylized Multi-group Binary Choice with Social
Interaction. J Stat Phys 171, 106--126 (2018).

\bibitem{Temperley}Temperley, H.N.V.: The Mayer Theory of Condensation
Tested against a Simple Model of the Imperfect Gas, Proc. Phys. Soc.,
A 67, pp. 233-238 (1954)

\bibitem{Thompson}Thompson, C.J.: Mathematical Statistical Mechanics,
Macmillan (1972)

\bibitem{Toth}Toth, Gabor: Correlated Voting in Multipopulation Models,
Two-Tier Voting Systems, and the Democracy Deficit, PhD Thesis, FernUniversit\"at
in Hagen. (2020) https://ub-deposit.fernuni-hagen.de/receive/mir\_mods\_00001617

\bibitem{wong} Wong, R.: Asymptotic Approximation of Integrals, SIAM 2001
\end{thebibliography}
\end{document}